\documentclass[review]{elsarticle}
\usepackage{color}
\usepackage{lineno,hyperref}
\usepackage{amsfonts}
\usepackage{amsmath}
\usepackage{amsthm}
\usepackage{amssymb}
\usepackage{amsthm,amsmath,amssymb}
\usepackage{mathrsfs}
\usepackage{amscd,multirow}
\usepackage{ulem}

\usepackage[caption=false,font=normalsize,labelfont=sf,textfont=sf]{subfig}
\usepackage{stfloats}

\usepackage{graphicx}
\usepackage{epstopdf}
\usepackage[titletoc]{appendix}
\usepackage{subeqnarray,cases}
\usepackage[ruled]{algorithm2e}
\newtheorem{theorem}{Theorem}[section]
\newtheorem{corollary}{Corollary}[section]
\newtheorem{lemma}{Lemma}[section]

\newtheorem{remark}{Remark}[section]

\newtheorem{assumption}{Assumption}[section]

\modulolinenumbers[5]

\begin{document}
	\begin{frontmatter}
		
		\title{ Strong asymptotic convergence of a slowly damped inertial primal-dual dynamical system controlled by a Tikhonov regularization term \tnoteref{mytitlenote}}
		
		
		\author[mymainaddress]{T.-T. Zhu}
		\ead{zttsicuandaxue@126.com}
		
		\author[mysecondaryaddress]{R. Hu}
		\ead{ronghumath@aliyun.com}
		
		\author[mymainaddress]{Y.-P. Fang\corref{mycorrespondingauthor}}
		\cortext[mycorrespondingauthor]{Corresponding author}
		\ead{ypfang@scu.edu.cn}
		
		\address[mymainaddress]{Department of Mathematics, Sichuan University, Chengdu, Sichuan, P.R. China}
		\address[mysecondaryaddress]{Department of Applied Mathematics, Chengdu University of Information Technology, Chengdu, Sichuan, P.R. China}
		
		\begin{abstract}
			We propose a slowly damped inertial primal-dual dynamical system controlled by a  Tikhonov regularization term, where the inertial term is introduced only for the primal variable,  for the linearly constrained convex optimization problem in a Hilbert space. Under mild conditions on the underlying parameters, by  a Lyapunov analysis approach, we prove the strong asymptotic convergence of  the trajectory of the proposed dynamic  to the minimal norm element of the  primal-dual solution set of the  problem, along with convergence rate results for  the primal-dual gap, the objective residual and the feasibility violation. We perform some numerical experiments to illustrate the theoretical findings.
		\end{abstract}
		
		\begin{keyword}
		 Linearly constrained convex optimization problem\sep  Inertial primal-dual dynamical system\sep Tikhonov regularization\sep Slow damping \sep Strong convergence \sep Minimal norm primal-dual solution.				
		\end{keyword}		
	\end{frontmatter}
	
	\section{Introduction} 

	Let $\mathcal{X}$ and $\mathcal{Y}$ be two real Hilbert spaces with the inner product $\langle \cdot, \cdot\rangle$ and the associated norm $\|\cdot\|$. The norm of the Cartesian product $\mathcal{X}\times\mathcal{Y}$ is defined by
	$$ \|(x,y)\|=\sqrt{\|x\|^2+\|y\|^2}$$
	for any $(x,y)\in \mathcal{X}\times\mathcal{Y}$.
	Let $f: \mathcal{X}\rightarrow \mathbb{R}$ be a continuously differentiable convex function, $A: \mathcal{X}\rightarrow \mathcal{Y}$ be a continuous linear operator and $b\in\mathcal{Y}$. Consider the linear equality constrained convex optimization problem
	\begin{equation}\label{z1}
		\min_{x\in\mathcal{X}} f(x), \quad \text{ s.t. }  \  Ax = b.
	\end{equation}
	Problem \eqref{z1} is a basic model for many important applications arising in machine learning, image recovery, network optimization and the energy dispatch of power grids. See e.g. \cite{ZLinandLiHandFang(2020),GoldsteinTandDonoghue(2014),ZengXLandLeiJLandChenJ(2022), PYiandHongYandLiu(2015)}. When $A=0$ and $b=0$, problem \eqref{z1} reduces to the unconstrained convex optimization problem
	\begin{equation}\label{zfg1}
		\min_{x\in\mathcal{X}}  \quad f(x).
	\end{equation} 
To solve   problem \eqref{zfg1}, a powerful approach is to follow the trajectory of the following inertial dynamical system
\begin{eqnarray}\label{IGS}
			\ddot{x}(t)+\gamma(t) \dot{x}(t)+\nabla f(x(t))=0,\quad\forall t\ge t_0,
	\end{eqnarray}	
where  $\gamma:[t_0,+\infty)\to [0,+\infty)$ is a continuous damping function and $t_0>0$. The tuning of the damping function $\gamma(t)$ plays a central role for establishing the minimization properties of the trajectory of \eqref{IGS}.  Cabot et al. \cite{Cabot} proved that the condition $\int_{t_0}^{+\infty} \gamma(t)dt=+\infty$ guarantees that the energy function $f$ along the trajectory $x(t)$ of  \eqref{IGS} converges toward its minimum.  The  case $\gamma(t)=\frac{\alpha}{t^q}$ with $0\le q\le 1$ and $\alpha>0$ is  of particular interest and importance in the literature. In this case, system \eqref{IGS} becomes
	\begin{eqnarray*}
		\text{(IGS)}_{q}\quad\quad
		\ddot{x}(t)+\frac{\alpha}{t^{q}} \dot{x}(t)+\nabla f(x(t))=0,\quad\forall t\ge t_0,
	\end{eqnarray*}
where $\frac{\alpha}{t^{q}}$ denotes a slow damping which  cannot decay rapidly to zero. Especially, when $q=0$, $\text{(IGS)}_{q}$ becomes  the classical heavy ball with friction system due to Polyak \cite{Polyak(1964)} which  models  the motion of a material point with unit mass under the action of the gravity force, the reaction force, and the friction force, subjected to stay on the graph of $f$. The convergence properties of  the heavy ball with friction system  were  investigated in \cite{Alvarezon(2000), Begoutbj(2015)}. When $q=1$, $\text{(IGS)}_{q}$ becomes the known inertial dynamical system proposed by Su et al. \cite{SuBoydandCandes(2016)} for understanding the acceleration of the Nesterov's accelerated algorithm \cite{Nesterov(1983),Nesterov(2013)}, and its convergence properties were intensively studied in  \cite{SuBoydandCandes(2016),AttouchCPR2018,May2017,AttouchCRR2019,Vassilis2018,Aujol2019}. Convergence rate results for $\text{(IGS)}_{q}$ with $0<q<1$ can be found in \cite{Cabot,Cabjde,Hara,Balti2017,Attouchcabot(2017),Sebb,GeB}.

Meanwhile, inertial dynamical systems controlled by Tikhonov regularization  terms have been developed  to  find the minimal norm solution of the problem under consideration, which is particular imprtant in practical situations. The following Tikhonov regularized inertial dynamical system
	\begin{eqnarray*}
		\text{(IGS)}_{q,\epsilon}\quad\quad
		\ddot{x}(t)+\frac{\alpha}{t^{q}} \dot{x}(t)+\nabla f(x(t))+\epsilon(t)x(t)=0
	\end{eqnarray*}
has been proposed in the literature to find  the minimal norm solution of  problem \eqref{zfg1}, where  $\epsilon:[t_0,+\infty)\to[0,+\infty)$ satisfying $\lim_{t\to+\infty}\epsilon(t) =0$  acts as a control on the trajectory, named as the Tikhonov regularization coefficient. With the additional term $\epsilon(t)x(t)$, compared with  $\text{(IGS)}_{q}$, the  trajectory of $\text{(IGS)}_{q,\epsilon}$ can converge strongly to the minimal norm solution of  problem \eqref{zfg1} under a suitable condition on the control parameter $\epsilon(t)$.  Under the condition  $\int_{t_0}^{+\infty}\epsilon(t)dt=+\infty$,  Attouch and Czarnecki \cite{AttouchandCzarnecki(2002)} proved  that  the trajectory of $\text{(IGS)}_{q,\epsilon}$ with $q=0$ converges strongly to the minimal norm solution of problem \eqref{zfg1}.  When $\epsilon(t)$ converges slowly to zero and $\alpha\ge 3$, Attouch et al. \cite{AttouchZH2018} established the strong convergence  in the inferior limit sense of the trajectory $x(t)$ of  $\text{(IGS)}_{q,\epsilon}$ with $q=1$ to  the minimal norm solution $x^*$ of problem \eqref{zfg1}, i.e., $\lim\inf_{t\rightarrow+\infty}\|x(t)-{x}^*\|=0$. Attouch and L{\'a}szl{\'o} \cite{Attouchlaszlo2021} showed  the strong convergence  in the inferior limit sense of the trajectory of $\text{(IGS)}_{q,\epsilon}$ with $\epsilon(t)=\frac{1}{t^{2q}}$ and $\frac{1}{3}<q<1$ to the minimal norm solution of  problem \eqref{zfg1}. Attouch et al. \cite{AttouchBCR2022311}   proved  the strong convergence of the trajectory $x(t)$ of  $\text{(IGS)}_{q,\epsilon}$ with $\epsilon(t)=\frac{1}{t^{2q}}$ and  $0<p<q+1$ to the minimal norm solution $x^*$ of   problem \eqref{zfg1} and derived the fast convergence rate $f(x(t))-\min f=\mathcal{O}\left(\frac{1}{t^{2q}}\right)$,  improving the result of \cite{Attouchlaszlo2021}. L{\'a}szl{\'o} \cite{Laszlo2023} further obtained the strong convergence of the trajectory $x(t)$ of  $\text{(IGS)}_{q,\epsilon}$ with  $\epsilon(t)=\frac{c}{t^{p}}$,  $0<q<1$ and $0<p<q+1$ to the minimal norm solution $x^*$ of  problem \eqref{zfg1}, along with  fast convergence rate results. For more results on this topic, we refer the reader to   \cite{botcsernek2021, Alecsalaszlo2021}.

In recent years, some inertial primal-dual dynamical systems  were developed for  the linear equality constrained convex optimization problem \eqref{z1}.  Zeng et al. \cite{ZengXLandLeiJLandChenJ(2022)} proposed the first inertial primal-dual dynamical system in the literature,  which is formulated as 
\begin{eqnarray*}
	\text{(Z-AVD)}\quad
	\begin{cases}
		\ddot{x}(t)+\frac{\alpha}{t}\dot{x}(t)&=-\nabla_x\mathcal{L}^{\rho}(x(t),\lambda(t)+\theta t\dot{\lambda}(t)),\\
		\ddot{\lambda}(t)+\frac{\alpha}{t}\dot{\lambda}(t)&=\nabla_{\lambda}\mathcal{L}^{\rho}(x(t)+\theta t\dot{x}(t),\lambda(t)),
	\end{cases}\quad\forall t\ge t_0,
\end{eqnarray*}
where $\alpha>0,$ $\theta>0$ and  $\mathcal{L}^{\rho}(x, \lambda)$ is the augmented Lagrangian function of problem \eqref{z1} with the penalty parameter  $\rho\ge 0$. Zeng et al. \cite{ZengXLandLeiJLandChenJ(2022)} proved  fast convergence rates for the primal-dual gap and the feasibility violation along the trajectory of $\text{(Z-AVD)}$,  extending the work of Su et al. \cite{SuBoydandCandes(2016)} from the unconstrained problem \eqref{zfg1} to the linearly constrained problem \eqref{z1}. Motivated by the work of  Zeng et al. \cite{ZengXLandLeiJLandChenJ(2022)},  He et al. \cite{HeHuFangetal(2021)} and Attouch et al. \cite{AttouchADMM(2022)}  proposed inertial primal-dual dynamical systems with general time-dependent dampings for solving problem \eqref{z1} with a separable structure.   Bo{\c t} and Nguyen \cite{BNguyen2022} improved the convergence rate results of Zeng et al. \cite{ZengXLandLeiJLandChenJ(2022)}, and  proved the  weak convergence  of the trajectory to a primal-dual optimal solution of problem \eqref{z1}, which is the first work on the weak convergence of the trajectory  in the literature. He et al. \cite{HHFIPD2023} further discussed the  convergence rate analysis of the following  inertial primal-dual dynamical system
	\begin{eqnarray*}
		\text{(He-ODE)}\quad
		\begin{cases}
			\ddot{x}(t)+\frac{\alpha}{t^q}\dot{x}(t)&=-\beta(t)\nabla_x\mathcal{L}^{\rho}(x(t),\lambda(t)+\theta t^{\kappa}\dot{\lambda}(t))+\varepsilon(t),\\
			\ddot{\lambda}(t)+\frac{\alpha}{t^q}\dot{\lambda}(t)&=\beta(t)\nabla_{\lambda}\mathcal{L}^{\rho}(x(t)+\theta t^{\kappa}\dot{x}(t),\lambda(t)), 
		\end{cases}\quad\forall t\ge t_0,
	\end{eqnarray*}
	where $0\leq q\leq \kappa\leq1$, $\beta: [t_0, +\infty)\to (0, +\infty)$ is a scaling coefficient and $\varepsilon: [t_0, +\infty)\to \mathcal{X}$ acts as a perturbation term. It is worth noticing that inertial primal-dual dynamical systems considered in \cite{ZengXLandLeiJLandChenJ(2022),HeHuFangetal(2021), AttouchADMM(2022),BNguyen2022,HHFIPD2023} have a same second-order plus second-order structure,  which involve inertial terms for both the primal and dual variables. 

Different from the inertial primal-dual dynamical systems considered in \cite{ZengXLandLeiJLandChenJ(2022),HeHuFangetal(2021), AttouchADMM(2022),BNguyen2022,HHFIPD2023}, He et al. \cite{HeHFetal(2022),HeHFiietal(2022)}  developed some inertial primal-dual dynamical systems with a second-order plus first-order structure, where the inertial term is introduced only for the primal variable,  for solving problem \eqref{z1}.
The first second-order plus first-order inertial primal-dual dynamic in the literature is formulated as
\begin{eqnarray*}
		\begin{cases}
		\ddot{x}(t)+\alpha\dot{x}(t)& = -\beta(t)\nabla_x \mathcal{L}^{\rho}(x(t),\lambda(t)),\\
		 \dot{\lambda}(t) &=\beta(t)\nabla_{\lambda}\mathcal{L}^{\rho}(x(t)+\theta\dot{x}(t),\lambda(t)),
	\end{cases}
\quad\forall t\ge 0,
	\end{eqnarray*}
which was proposed  by He et al. \cite{HeHFetal(2022)} and can be viewed as an extension of  Polyak's  heavy ball with friction system \cite{Polyak(1964)}. He et al. also \cite{HeHFiietal(2022)} proposed  and studied the following  second-order plus first-order inertial primal-dual dynamical system with a vanishing damp
	\begin{eqnarray*}
				\begin{cases}
			\ddot{x}(t)+\frac{\alpha}{t}\dot{x}(t)&=-\beta(t)\nabla_x \mathcal{L}(x(t),\lambda(t))
			+\varepsilon(t),\\
			\dot{\lambda}(t)&=t\beta(t)\nabla_{\lambda} \mathcal{L}(x(t)+\frac{t}{\alpha-1}\dot{x}(t),\lambda(t)), 
		\end{cases}\quad\forall t\ge t_0,
	\end{eqnarray*}
which can be regarded as an extension of the inertial dynamical system due to Su et al. \cite{SuBoydandCandes(2016)},	where $\alpha>1$ and $\mathcal{L}(x,\lambda)$ is the  Lagrangian function of problem \eqref{z1}.

Recently, some researchers started to investigate  inertial primal-dual dynamical systems controlled by Tikhonov regularization terms for the linear equality constrained convex optimization problem \eqref{z1}.  The first Tikhonov regularized inertial primal-dual dynamical system was proposed by Zhu et al. \cite{zhuhufang1}, which is formulated as
	\begin{eqnarray}\label{zhf1}
		\begin{cases}
			\ddot{x}(t)+\frac{\alpha}{t}\dot{x}(t)&=-\nabla_x \mathcal{L}^{\rho}(x(t),\lambda(t))-\epsilon(t)x(t),\\
			\dot{\lambda}(t)&=t\nabla_{\lambda} \mathcal{L}^{\rho}(x(t)+\dfrac{t}{\alpha-1}\dot{x}(t),\lambda(t)),
		\end{cases}\quad\forall t\ge t_0.
	\end{eqnarray}
Under the conditions that  $\lim_{t\to+\infty}t^2\epsilon(t)= +\infty$ and $\int_{t_0}^{+\infty}\frac{\epsilon(t)}{t}dt<+\infty$, Zhu et al. \cite{zhuhufang1} proved  the strong convergence of the primal trajectory $x(t)$ of \eqref{zhf1}  to the minimal norm solution   $x^*$ of problem \eqref{z1}  in the inferior limit sense, i.e.,  $\lim\inf_{t\rightarrow+\infty}\|x(t)-x^*\|=0$. By introducing the scaling term and  the Tikhonov regularization term into $\text{(Z-AVD)}$,  Zhu et al. \cite{zhuhufang2} also proposed the following  Tikhonov regularized inertial primal-dual dynamical system
\begin{eqnarray}\label{zhu22}
	\begin{cases}
		\ddot{x}(t)+\frac{\alpha}{t}\dot{x}(t)&=-\beta(t)\left(\nabla_x\mathcal{L}^{\rho}(x(t),\lambda(t)+\theta t\dot{\lambda}(t)) +\epsilon(t)x(t)\right),\\
		\ddot{\lambda}(t)+\frac{\alpha}{t}\dot{\lambda}(t)&=\beta(t)\nabla_{\lambda}\mathcal{L}^{\rho}(x(t)+\theta t\dot{x}(t),\lambda(t)).
	\end{cases}
\end{eqnarray}
Under the following  conditions
\begin{eqnarray*}
	t\dot{\beta}(t)\leq\frac{1-2\theta}{\theta}\beta(t), \int_{t_{0}}^{+\infty}\frac{\beta(t)\epsilon(t)}{t}dt<+\infty, \lim_{t\to+\infty}t^2\beta(t)\epsilon(t)= +\infty,
	\end{eqnarray*}
Zhu et al. \cite{zhuhufang2} proved the strong convergence   in the inferior limit sense  of the primal trajectory of  \eqref{zhu22} to the minimal norm solution of problem \eqref{z1}.  Let us emphasize that  only the strong convergence  in the inferior limit sense of the primal trajectory $x(t)$ to the minimal norm solution $x^*$ was established in \cite{zhuhufang1,zhuhufang2} because the primal-dual dynamic under consideration involves the Tikhonov regularization term only for the primal variable.  Very recently,  Chbani et al. \cite{ChbaniRBOn(2024)} proposed the following Tikhonov regularized primal-dual dynamical system with constant damping    
	\begin{eqnarray}\label{z21}
		\begin{cases}
			\ddot{x}(t)+\alpha\dot{x}(t)+t^p\nabla_{x}\mathcal{L}(x(t), \lambda(t))+cx(t)&=0,\\
			\dot{\lambda}(t)-t^{p}\nabla_{\lambda}\mathcal{L}(x(t)+\theta\dot{x}(t), \lambda(t))+c\lambda(t)&=0,
		\end{cases}
	\end{eqnarray}
where $\alpha>0$, $0<p<1$, $c>0$ and $\theta>0$, which can be viewed as  a Tikhonov regularization variant of  the inertial primal-dual dynamic due to He et al. \cite{HeHFetal(2022)}. Notice that system \eqref{z21} involves  the Tikhonov regularization terms  for both the primal and dual variables.
Under suitable conditions,  Chbani et al. \cite{ChbaniRBOn(2024)} proved that the trajectory $(x(t),\lambda(t))$ of  \eqref{z21} converges strongly to the minimal norm primal-dual solution $(x^*,\lambda^*)$ of  problem \eqref{z1}, i.e.,
$$\lim_{t\rightarrow+\infty}\|(x(t), \lambda(t))-(x^*,\lambda^*)\|=0,$$
 along with convergence rate results of the primal-dual gap, the objective residual and the feasibility violation. It is worth mentioning that the proofs of \cite[Theorem 3.2]{ChbaniRBOn(2024)} on convergence rates for  the objective residual and the feasibility violation were based on \cite[Lemma 6]{HeHFiietal(2022)} (\cite[Lemma 2.1]{ChbaniRBOn(2024)}),  which cannot be  applied there since the  function $a(s)$ is dependent on $t$.  In this paper, we  consider the following  slowly damped  inertial primal-dual dynamical system controlled by  a Tikhonov regularation term 
\begin{eqnarray}\label{z2}
		\begin{cases}
			\ddot{x}(t)+\frac{\alpha}{t^q}\dot{x}(t)+t^s\left(\nabla_{x}\mathcal{L}(x(t), \lambda(t))+\frac{c}{t^p}x(t)\right)&=0,\\
			\dot{\lambda}(t)-t^{q+s}\left(\nabla_{\lambda}\mathcal{L}(x(t)+\theta t^q\dot{x}(t), \lambda(t))-\frac{c}{t^p}\lambda(t)\right)&=0,
		\end{cases}
	\end{eqnarray}
	where $t\geq t_0>0$, $0\leq q<1$,  $0<p<1$, $c>0$, $\alpha>0$, $\theta>0$ and $s$ is a constant.  This inertial primal-dual dynamic is more general than system \eqref{z2}. Indeed, when $q=0$ and $s=p$, system \eqref{z2} becomes system \eqref{z21} considered by  Chbani et al. \cite{ChbaniRBOn(2024)}. Under mild conditions on the parameters $q$, $p$ and $s$, we shall establish the convergence rate results for the primal-dual gap, the objective residual and the feasibility violation, and the strong convergence of the trajectory $(x(t),\lambda(t))$ of \eqref{z2} to the minimum norm element $(x^*,\lambda^*)$ of the primal-dual optimal solution set of problem \eqref{z1}. Let us emphasize that we develop two new lemmas (Lemma  \ref{lemma2.2} and Lemma  \ref{lemma2.3})  to establish the convergence rates of the objective residual and the feasibility violation. In terms of Lemma  \ref{lemma2.2} and Lemma  \ref{lemma2.3}, we fix the problem that  \cite[Lemma 6]{HeHFiietal(2022)} cannot be applied in the proof of \cite[Theorem 3.2]{ChbaniRBOn(2024)}. Our main contributions are summarized as follows:

 {\bf (a)}. We  propose an inertial second-order plus first-order primal-dual dynamical system, controlled by a  Tikhonov regularization term, with a slow damping   $\frac{\alpha}{t^q}$, where $0\le q<1$, for  the linearly constrained convex optimization problem  \eqref{z1}. Our dynamic  \eqref{z2} is very general and includes the recent  Tikhonov regularized inertial dynamic considered by  Chbani et al. \cite{ChbaniRBOn(2024)} as a special case.  Compared to the inertial primal-dual dynamics considered in \cite{zhuhufang1,zhuhufang2}, the proposed dynamic \eqref{z2} involves  Tikhonov regularization terms  for both the primal and dual variables.

{\bf (b) } Under mild conditions on the underlying parameters, we prove  the strong convergence  of the trajectory of \eqref{z2}  to the minimum norm element of the primal-dual optimal solution set of problem \eqref{z1}, along with  the convergence rate results for the primal-dual gap, the objective residual and the feasibility violation. Let us emphasize that  the proofs of the objective residual and the feasibility violation are based on two new  developed lemmas (Lemma  \ref{lemma2.2} and Lemma  \ref{lemma2.3}) which fix the problem that  \cite[Lemma 6]{HeHFiietal(2022)} cannot be applied in the proof of \cite[Theorem 3.2]{ChbaniRBOn(2024)}. Let us also emphasize that  our strong convergence result is $\lim_{t\rightarrow+\infty}\|(x(t), \lambda(t))-(x^*,\lambda^*)\|=0$ instead of  the result  $\lim\inf_{t\rightarrow+\infty}\|x(t)-x^*\|=0$ established in  \cite{zhuhufang1,zhuhufang2}, where $x(t)$ and $y(t)$ denote respectively the primal and dual trajectories of the primal-dual dynamical system under consideration, and $(x^*,\lambda^*)$ denotes the minimal norm  primal-dual solution.

The rest of this paper is organized as follows: In Section 2, we provide some preliminary results which will be used in convergence analysis. In Section 3,  we investigate the convergence properties of the primal-dual gap, the objective function value and the feasibility violation, and the strong convergence of the primal-dual trajectory generated by system \eqref{z2}. Finally,  we perform  in Section 4 some numerical experiments to illustrate our theoretical findings.
	
	\section{Preliminary results}

Throughout this paper, we will make the following standard assumption on the parameters  and functions in  problem \eqref{z1} and system \eqref{z2}:

\begin{assumption}\label{AS-F}
	Suppose that  $f: \mathcal{X}\rightarrow \mathbb{R}$ is a continuously differentiable convex function, $A: \mathcal{X}\rightarrow \mathcal{Y}$ is a continuous linear operator,  the primal-dual solution set  $\Omega$ of problem \eqref{z1} is nonempty, and

 $$\alpha>0,\quad \theta>\frac{1}{\alpha},\quad 0\leq q<1,\quad 0<p<1-q,\quad c>0.$$
\end{assumption}

	Recall that  the Lagrangian function $\mathcal{L}: \mathcal{X}\times\mathcal{Y}\rightarrow\mathbb{R}$ of problem \eqref{z1} is defined by
	\begin{eqnarray*}
		\mathcal{L}(x,\lambda)=f(x)+\langle \lambda, Ax-b\rangle.
	\end{eqnarray*}
and that  $(\hat{x},\hat{\lambda})\in  \mathcal{X}\times\mathcal{Y}$ is called a saddle point of $\mathcal{L}$  if and only if 
	$$\mathcal{L}(\hat{x}, \lambda)\leq \mathcal{L}(\hat{x}, \hat{\lambda})\leq\mathcal{L}(x, \hat{\lambda}), \quad \forall (x,\lambda)\in\mathcal{X}\times\mathcal{Y}.$$
The  saddle point set of $\mathcal{L}$ is denote by $\Omega$. It is well-known that  $(\hat{x},\hat{\lambda})\in\Omega$ if and only if
	\begin{eqnarray}\label{zc1}
		\begin{cases}
			\nabla f(\hat{x})+A^T\hat{\lambda}=0,\\
			A\hat{x}-b=0.
		\end{cases}
	\end{eqnarray}
A pair  $(\hat{x},\hat{\lambda})\in\Omega$ is also called a primal-dual solution of  problem \eqref{z1}.  
	
	Define $\mathcal{L}_{t} : \mathcal{X}\times\mathcal{Y}\rightarrow\mathbb{R}$  by
	\begin{eqnarray}\label{z3}
		\mathcal{L}_{t}(x,\lambda)&=&\mathcal{L}(x,\lambda)+\frac{c}{2t^p}\left(\|x\|^2-\|\lambda\|^2\right)\\
		&=&f(x)+\langle \lambda, Ax-b\rangle+\frac{c}{2t^p}\left(\|x\|^2-\|\lambda\|^2\right).\nonumber
	\end{eqnarray}
Clearly, $\mathcal{L}_{t}(\cdot,\lambda)$ is $\frac{c}{t^p}$-strongly convex and    $\mathcal{L}_{t}(x,\cdot)$ is $\frac{c}{t^p}$-strongly concave for every $(x,\lambda)\in  \mathcal{X}\times\mathcal{Y}$. Consequently, $\mathcal{L}_{t}$ has a unique saddle point. Set 
	$$(x_t,\lambda_t) :=\arg\min\max_{\mathcal{X}\times\mathcal{Y}}\mathcal{L}_{t}(x, \lambda).$$
	Then,
	\begin{eqnarray}\label{zv1}
		\mathcal{L}_{t}(x_t, \lambda)\leq\mathcal{L}_{t}(x_t, \lambda_t)\leq\mathcal{L}_{t}(x, \lambda_t), \quad \forall (x, \lambda)\in\mathcal{X}\times\mathcal{Y}.
	\end{eqnarray}
Using the first-order optimality condition, we have
	\begin{eqnarray}\label{z4}
		\begin{cases}
			0&=\nabla_{x}\mathcal{L}_{t}(x_t, \lambda_t)=\nabla f(x_t)+A^T\lambda_t+\frac{c}{t^p}x_t,\\
			0&=\nabla_{\lambda}\mathcal{L}_{t}(x_t, \lambda_t)=Ax_t-b-\frac{c}{t^p}\lambda_t.
		\end{cases}
	\end{eqnarray} 

The following lemmas play  crucial roles in establishing  convergence results.

	\begin{lemma}\label{lemma2.1}\cite[Lemma 2.3]{ChbaniRBOn(2024)}
		Let $(x^*,\lambda^*)$ be the minimum norm element of the primal-dual optimal solution set $\Omega$ of problem \eqref{z1}.  Then, it holds:
		\begin{itemize}
			\item[(i)] $\lim_{t\rightarrow+\infty}\|(x_t,\lambda_t)-(x^*,\lambda^*)\|=0$ and $\|(x_t,\lambda_t)\|\leq \|(x^*,\lambda^*)\|$  for all $t\geq t_0$.
			\item[(ii)] $\|(\dot{x}_t, \dot{\lambda}_t)\|\leq\frac{p}{t}\|(x_t,\lambda_t)\|\leq\frac{p}{t}\|(x^*,\lambda^*)\|$  for all $t\geq t_0$.
		\end{itemize}
	\end{lemma}
	
	\begin{lemma}\label{lemma012}\cite[Lemma 2.4]{ChbaniRBOn(2024)}
		For any $t\geq t_0$, it holds
		\begin{eqnarray*}
			\frac{d}{dt}\mathcal{L}_{t}(x_t, \lambda_t)=\frac{cp}{2t^{p+1}}\left(\|\lambda_t\|^2-\|x_t\|^2\right).
		\end{eqnarray*}
	\end{lemma}
	
The following lemma generalizes  \cite[Lemma 6]{HeHFiietal(2022)}. 
	\begin{lemma}\label{lemma2.2}
		 Let $\delta>0$,  $\mu\geq0$ and $\nu\geq0$. Suppose that  $g:[\delta,+\infty)\rightarrow \mathcal{X}$ and $a : [\delta,+\infty)\rightarrow[0,+\infty)$ are two continuously differentiable functions.  If there exists a constant $C\geq0$ such that
		\begin{eqnarray}\label{zfe2}
			\left\|g(t)+e^{-\mu t^{\nu}}\int_{\delta}^{t}a(\tau)g(\tau)d\tau\right\|\leq C, \quad \forall t\geq \delta,
		\end{eqnarray}
		then
		$$\sup_{t\geq \delta}\|g(t)\|<+\infty.$$
	\end{lemma}
	
	\begin{proof}
		Let's define $G : [\delta, +\infty)\rightarrow  \mathcal{X}$ by 
		\begin{eqnarray}\label{zfe1}
			G(t)=e^{\int_{\delta}^{t}a(\tau)e^{-\mu\tau^{\nu}}d\tau}\int_{\delta}^{t}a(\tau)g(\tau)d\tau.
		\end{eqnarray}
		Combining \eqref{zfe2} and \eqref{zfe1}, we get
		\begin{eqnarray*}
			\left\|\dot{G}(t)\right\|&=&\left\|a(t)e^{\int_{\delta}^{t}a(\tau)e^{-\mu\tau^{\nu}}d\tau}e^{-\mu t^{\nu}}\int_{\delta}^{t}a(\tau)g(\tau)d\tau+a(t)g(t)e^{\int_{\delta}^{t}a(\tau)e^{-\mu\tau^{\nu}}d\tau}\right\|\\
			&=&a(t)e^{\int_{\delta}^{t}a(\tau)e^{-\mu\tau^{\nu}}d\tau}\left\|g(t)+e^{-\mu t^{\nu}}\int_{\delta}^{t}a(\tau)g(\tau)d\tau\right\|\\
			&\leq&Ca(t)e^{\int_{\delta}^{t}a(\tau)e^{-\mu\tau^{\nu}}d\tau}
		\end{eqnarray*}
		 for all $t\geq \delta$.
		Observe that  $G(\delta)=0$. It follows that
		\begin{eqnarray}\label{abf1}
			\left\|G(t)\right\|=\left\|\int_{\delta}^{t}\dot{G}(w)dw\right\|\leq\int_{\delta}^{t}\left\|\dot{G}(w)\right\|dw\leq C\int_{\delta}^{t}a(w)e^{\int_{\delta}^{w}a(\tau)e^{-\mu\tau^{\nu}}d\tau}dw.
		\end{eqnarray}
		Since
		\begin{eqnarray*}
			\frac{d}{dw}\left(e^{\mu w^{\nu}}e^{\int_{\delta}^{w}a(\tau)e^{-\mu\tau^{\nu}}d\tau}\right)&=&\mu \nu w^{\nu-1}e^{\mu w^{\nu}}e^{\int_{\delta}^{w}a(\tau)e^{-\mu\tau^{\nu}}d\tau}+a(w)e^{\int_{\delta}^{w}a(\tau)e^{-\mu\tau^{\nu}}d\tau}\\
			&\geq&a(w)e^{\int_{\delta}^{w}a(\tau)e^{-\mu\tau^{\nu}}d\tau},
		\end{eqnarray*}
		we get
		\begin{eqnarray*}
			\int_{\delta}^{t}a(w)e^{\int_{\delta}^{w}a(\tau)e^{-\mu\tau^{\nu}}d\tau}dw\leq\int_{\delta}^{t}d\left(e^{\mu w^{\nu}}e^{\int_{\delta}^{w}a(\tau)e^{-\mu\tau^{\nu}}d\tau}\right)=e^{\mu t^{\nu}}e^{\int_{\delta}^{t}a(\tau)e^{-\mu\tau^{\nu}}d\tau}-e^{\mu\delta^{\nu}}.
		\end{eqnarray*}
This together with \eqref{abf1} yields
		\begin{eqnarray*}
			\left\|G(t)\right\|\leq Ce^{\mu t^{\nu}}e^{\int_{\delta}^{t}a(\tau)e^{-\mu\tau^{\nu}}d\tau}-Ce^{\mu\delta^{\nu}},\quad\forall t\ge \delta.
		\end{eqnarray*}
Using \eqref{zfe1}, we have 
		\begin{eqnarray*}
			e^{-\mu t^{\nu}}\left\|\int_{\delta}^{t}a(\tau)g(\tau)d\tau\right\|\leq C-\frac{Ce^{\mu\delta^{\nu}}}{e^{\mu t^{\nu}}e^{\int_{\delta}^{t}a(\tau)e^{-\mu\tau^{\nu}}d\tau}}\leq C, \quad\forall t\ge \delta,
		\end{eqnarray*}
		which together with \eqref{zfe2} implies 
		$$\left\|g(t)\right\|\leq C+e^{-\mu t^{\nu}}\left\|\int_{\delta}^{t}a(\tau)g(\tau)d\tau\right\|\leq 2C<+\infty, \quad \forall t\geq \delta.$$
		\end{proof}

\begin{remark}  
When  $\mu=0$,  Lemma  \ref{lemma2.2} reduces to   \cite[Lemma 6]{HeHFiietal(2022)}.
\end{remark}

	\begin{lemma}\label{lemma2.3}
		 Let $\delta>0$,  $\mu\geq0$ and $\nu\geq0$. Suppose that  $g:[\delta,+\infty)\rightarrow  \mathcal{X}$ and  $a : [\delta,+\infty)\rightarrow (-\infty, 0]$  are two continuously differentiable functions. If  there exist constants  $C_0\in (-1, 0)$ and  $\widehat{C}\geq0$ such that
		\begin{eqnarray}\label{zfe4}
			e^{-\mu t^{\nu}}\int_{\delta}^{t}a(\tau)d\tau\geq C_0, \quad \forall t\geq \delta
		\end{eqnarray}
		and
		\begin{eqnarray}\label{zfe3}
			\left\|g(t)+e^{-\mu t^{\nu}}\int_{\delta}^{t}a(\tau)g(\tau)d\tau\right\|\leq \widehat{C}, \quad \forall t\geq \delta,
		\end{eqnarray}
		then
		$$\sup_{t\geq \delta}\|g(t)\|<+\infty.$$
	\end{lemma}
	
	\begin{proof}
		Define $G : [\delta, +\infty)\rightarrow  \mathcal{X}$ by 
		\begin{eqnarray*}
			G(t)=e^{\int_{\delta}^{t}a(\tau)e^{-\mu\tau^{\nu}}d\tau}\int_{\delta}^{t}a(\tau)g(\tau)d\tau.
		\end{eqnarray*}
		It follows from \eqref{zfe3}  that
		\begin{eqnarray*}
			\left\|\dot{G}(t)\right\|&=&\left\|a(t)e^{\int_{\delta}^{t}a(\tau)e^{-\mu\tau^{\nu}}d\tau}e^{-\mu t^{\nu}}\int_{\delta}^{t}a(\tau)g(\tau)d\tau+e^{\int_{\delta}^{t}a(\tau)e^{-\mu\tau^{\nu}}d\tau}a(t)g(t)\right\|\\
			&=&-a(t)e^{\int_{\delta}^{t}a(\tau)e^{-\mu\tau^{\nu}}d\tau}\left\|g(t)+e^{-\mu t^{\nu}}\int_{\delta}^{t}a(\tau)g(\tau)d\tau\right\|\\
			&\leq&-\widehat{C}a(t)e^{\int_{\delta}^{t}a(\tau)e^{-\mu\tau^{\nu}}d\tau}, \quad \forall t\geq \delta.
		\end{eqnarray*}
		According to the definition of $G(t)$, we have $G(\delta)=0$. Then, it holds 
		\begin{eqnarray}\label{zfe5}
			\left\|G(t)\right\|=\left\|\int_{\delta}^{t}\dot{G}(w)dw\right\|\leq\int_{\delta}^{t}\left\|\dot{G}(w)\right\|dw\leq-\widehat{C}\int_{\delta}^{t}a(w)e^{\int_{\delta}^{w}a(\tau)e^{-\mu\tau^{\nu}}d\tau}dw.
		\end{eqnarray}
		By using \eqref{zfe4} and $a(t)\leq0$, we have  
		\begin{eqnarray*}
			\frac{d}{dw}\left(-\int_{\delta}^{w}a(\tau)d\tau e^{\int_{\delta}^{w}a(\tau)e^{-\mu\tau^{\nu}}d\tau}\right)&=&
			-e^{-\mu w^{\nu}}a(w)\int_{\delta}^{w}a(\tau)d\tau e^{\int_{\delta}^{w}a(\tau)e^{-\mu\tau^{\nu}}d\tau}\\
			&&-a(w)e^{\int_{\delta}^{w}a(\tau)e^{-\mu\tau^{\nu}}d\tau}\\
			&\geq&-\left(1+C_0\right)a(w)e^{\int_{\delta}^{w}a(\tau)e^{-\mu\tau^{\nu}}d\tau}, \quad \forall w\geq \delta.
		\end{eqnarray*}
		This  together with \eqref{zfe5} and  $C_0>-1$ implies that for any $t\geq \delta$, 
		\begin{eqnarray*}
			\left\|G(t)\right\|&\leq&-\frac{\widehat{C}}{1+C_0}\int_{\delta}^{t}a(\tau)d\tau e^{\int_{\delta}^{t}a(\tau)e^{-\mu\tau^{\nu}}d\tau}.
		\end{eqnarray*}
		By the definition of $G(t)$, we obtain 
		$$\left\|\int_{\delta}^{t}a(\tau)g(\tau)d\tau\right\|\leq-\frac{\widehat{C}}{1+C_0}\int_{\delta}^{t}a(\tau)d\tau.$$
		Using \eqref{zfe4} and \eqref{zfe3},  we have 
		$$\left\|g(t)\right\|\leq-\frac{\widehat{C}}{1+C_0}e^{-\mu t^{\nu}}\int_{\delta}^{t}a(\tau)d\tau+\widehat{C}\leq -\frac{\widehat{C}C_0}{1+C_0}+\widehat{C}<+\infty$$
for all $t\geq \delta$.
		Thus,
		$$\sup_{t\geq \delta}\|g(t)\|<+\infty.$$
	\end{proof}
	
\begin{remark}  
Lemma  \ref{lemma2.3}  can be viewed as a partial generalization of \cite[Lemma 3.3]{HeTLFconver(2023)}. Indeed, it has been shown in \cite[Lemma 3.3]{HeHFiietal(2022)} that    the conclusion of  Lemma  \ref{lemma2.3} with  $\mu=0$ holds without the assumption  $-1< C_0$.
\end{remark}

	\section{Convergence analysis}
In this section we shall investigate  the strong convergence of  the trajectory of system \eqref{z2} and   the convergence rates of the primal-dual gap, the objective residual and the feasibility violation. To do so, we need the following lemma.

	\begin{lemma}\label{lemma3.1}
		Assume that $\theta>\frac{1}{\alpha}$ and let $(x,\lambda): [t_{0},+\infty)\rightarrow \mathcal{X}\times\mathcal{Y}$ be a solution of \eqref{z2}. Define $\mathcal{E}:[t_0,+\infty)\to \mathbb{R}$ by
		\begin{eqnarray}\label{z8}
			\mathcal{E}(t)&=&\theta^2 t^{2q+s}\left(\mathcal{L}_{t}(x(t),\lambda_t)-\mathcal{L}_{t}(x_t,\lambda_t)\right)+\frac{1}{2}\|x(t)-x_t+\theta t^q\dot{x}(t)\|^2\nonumber\\
			&&+\frac{\alpha\theta-1-\theta q t^{q-1}}{2}\|x(t)-x_t\|^2+\frac{\theta}{2}\|\lambda(t)-\lambda_t\|^2.
		\end{eqnarray}
		Then, there exists $t_1\ge t_0$ such that
		\begin{eqnarray*}
			\dot{\mathcal{E}}(t)+\frac{K}{t^r}\mathcal{E}(t)&\leq&\theta t^{q+s}\left(\theta(2q+s)t^{q-1}-1+\theta Kt^{q-r}\right) \left(\mathcal{L}_{t}(x(t),\lambda_t)-\mathcal{L}_{t}(x_t,\lambda_t)\right)\\
			&&+\frac{1}{2}(\theta^2\|A\|^2t^{2q+s-1}+q(1-q)\theta t^{q-2}-c\theta t^{q+s-p}+\frac{(\alpha\theta-q\theta t^{q-1})t^{q+s-p}}{a_2}\\
			&&+(\alpha\theta+1-q\theta t^{q-1})Kt^{-r})\|x(t)-x_t\|^2\\
			&&+\theta t^q\left(1-\alpha\theta+\frac{1}{2a_1}+\theta qt^{q-1}+\theta Kt^{q-r}\right)\|\dot{x}(t)\|^2\\
			&&+\frac{\theta}{2}\left(\left(\frac{1}{a_3}-c\right)t^{q+s-p}+Kt^{-r}\right)\|\lambda(t)-\lambda_t\|^2\\
			&&+\frac{\theta}{2}\left(a_1 t^q+(\alpha-qt^{q-1})a_2t^{p-q-s}\right)\|\dot{x}_t\|^2+\frac{\theta}{2}\left(\theta t^{2q+s+1}+a_3t^{p-q-s}\right)\|\dot{\lambda}_t\|^2\\
			&&+\frac{cp\theta^2}{2}t^{2q+s-p-1}\left(\|x_t\|^2-\|x(t)\|^2\right) 
		\end{eqnarray*}
for all  $t\ge t_1$, where $K$, $r$, $a_1$, $a_2$ and $a_3$ are arbitrarily positive constants.
	\end{lemma}
	
	\begin{proof}
		By using \eqref{z8} and \eqref{z2}, we have 
		\begin{eqnarray}\label{z6}
			\dot{\mathcal{E}}(t)&=&\theta^2(2q+s) t^{2q+s-1}\left(\mathcal{L}_{t}(x(t),\lambda_t)-\mathcal{L}_{t}(x_t,\lambda_t)\right)+\theta^2t^{2q+s}\frac{d}{dt}\left(\mathcal{L}_{t}(x(t),\lambda_t)-\mathcal{L}_{t}(x_t,\lambda_t)\right)\nonumber\\
			&&+\langle x(t)-x_t+\theta t^q\dot{x}(t), \dot{x}(t)-\dot{x}_t+\theta qt^{q-1}\dot{x}(t)+\theta t^q\ddot{x}(t)\rangle+\frac{\theta q(1-q)t^{q-2}}{2}\|x(t)-x_t\|^2\nonumber\\
			&&+(\alpha\theta-1-\theta q t^{q-1})\langle x(t)-x_t, \dot{x}(t)-\dot{x}_t\rangle+\theta\langle \lambda(t)-\lambda_t, \dot{\lambda}(t)-\dot{\lambda}_t \rangle\nonumber\\
			&=&\theta^2(2q+s) t^{2q+s-1}\left(\mathcal{L}_{t}(x(t),\lambda_t)-\mathcal{L}_{t}(x_t,\lambda_t)\right)+\theta^2t^{2q+s}\frac{d}{dt}\left(\mathcal{L}_{t}(x(t),\lambda_t)-\mathcal{L}_{t}(x_t,\lambda_t)\right)\nonumber\\
			&&+\left\langle x(t)-x_t+\theta t^q\dot{x}(t), (1-\alpha\theta+\theta qt^{q-1})\dot{x}(t)-\dot{x}_t-\theta t^{q+s}\left(\nabla_{x}\mathcal{L}(x(t), \lambda(t))+\frac{c}{t^p}x(t)\right)\right\rangle\nonumber\\
			&&+\frac{\theta q(1-q)t^{q-2}}{2}\|x(t)-x_t\|^2+(\alpha\theta-1-\theta q t^{q-1})\langle x(t)-x_t, \dot{x}(t)-\dot{x}_t\rangle\nonumber\\
			&&+\theta t^{q+s}\left\langle \lambda(t)-\lambda_t, \nabla_{\lambda}\mathcal{L}(x(t)+\theta t^q\dot{x}(t), \lambda(t))-\frac{c}{t^p}\lambda(t)\right\rangle-\theta\langle \lambda(t)-\lambda_t, \dot{\lambda}_t \rangle\nonumber\\
			&=&\theta^2(2q+s) t^{2q+s-1}\left(\mathcal{L}_{t}(x(t),\lambda_t)-\mathcal{L}_{t}(x_t,\lambda_t)\right)+\theta^2t^{2q+s}\frac{d}{dt}\left(\mathcal{L}_{t}(x(t),\lambda_t)-\mathcal{L}_{t}(x_t,\lambda_t)\right)\nonumber\\
			&&+(1-\alpha\theta+\theta qt^{q-1})\langle x(t)-x_t, \dot{x}(t)\rangle+(1-\alpha\theta+\theta qt^{q-1})\theta t^q\|\dot{x}(t)\|^2-\langle x(t)-x_t, \dot{x}_t\rangle\nonumber\\
			&&-\theta t^q\langle\dot{x}(t), \dot{x}_t\rangle-\theta t^{q+s}\langle x(t)-x_t+\theta t^q\dot{x}(t), \nabla_{x}\mathcal{L}(x(t), \lambda(t))+\frac{c}{t^p}x(t)\rangle\nonumber\\
			&&+\frac{\theta q(1-q)t^{q-2}}{2}\|x(t)-x_t\|^2+(\alpha\theta-1-\theta q t^{q-1})\langle x(t)-x_t, \dot{x}(t)\rangle\nonumber\\
			&&-(\alpha\theta-1-\theta q t^{q-1})\langle x(t)-x_t, \dot{x}_t\rangle-\theta\langle \lambda(t)-\lambda_t, \dot{\lambda}_t \rangle\nonumber\\
			&&+\theta t^{q+s}\left\langle \lambda(t)-\lambda_t, \nabla_{\lambda}\mathcal{L}(x(t)+\theta t^q\dot{x}(t), \lambda(t))-\frac{c}{t^p}\lambda(t)\right\rangle\nonumber\\
			&=&\theta^2(2q+s) t^{2q+s-1}\left(\mathcal{L}_{t}(x(t),\lambda_t)-\mathcal{L}_{t}(x_t,\lambda_t)\right)+\theta^2t^{2q+s}\frac{d}{dt}\left(\mathcal{L}_{t}(x(t),\lambda_t)-\mathcal{L}_{t}(x_t,\lambda_t)\right)\nonumber\\
			&&+(1-\alpha\theta+\theta qt^{q-1})\theta t^q\|\dot{x}(t)\|^2-(\alpha\theta-\theta q t^{q-1})\langle x(t)-x_t, \dot{x}_t\rangle-\theta t^q\langle\dot{x}(t), \dot{x}_t\rangle\nonumber\\
			&&-\theta t^{q+s}\langle x(t)-x_t+\theta t^q\dot{x}(t), \nabla_{x}\mathcal{L}(x(t), \lambda(t))+\frac{c}{t^p}x(t)\rangle+\frac{\theta q(1-q)t^{q-2}}{2}\|x(t)-x_t\|^2\nonumber\\
			&&-\theta\langle \lambda(t)-\lambda_t, \dot{\lambda}_t\rangle+\theta t^{q+s}\left\langle \lambda(t)-\lambda_t, \nabla_{\lambda}\mathcal{L}(x(t)+\theta t^q\dot{x}(t), \lambda(t))-\frac{c}{t^p}\lambda(t)\right\rangle.
		\end{eqnarray}
Since
		\begin{eqnarray*}
			\nabla_{x}\mathcal{L}(x(t), \lambda(t))+\frac{c}{t^p}x(t)&=&\nabla_{x}\mathcal{L}(x(t), \lambda_t)+\frac{c}{t^p}x(t)+A^T(\lambda(t)-\lambda_t)\\
			&=&\nabla_{x}\mathcal{L}_{t}(x(t), \lambda_t)+A^T(\lambda(t)-\lambda_t)
		\end{eqnarray*}
		and 
		\begin{eqnarray*}
			\nabla_{\lambda}\mathcal{L}(x(t)+\theta t^q\dot{x}(t), \lambda(t))-\frac{c}{t^p}\lambda(t)&=&\nabla_{\lambda}\mathcal{L}(x_t, \lambda(t))-\frac{c}{t^p}\lambda(t)+A\left(x(t)-x_t+\theta t^q\dot{x}(t)\right)\\
			&=&\nabla_{\lambda}\mathcal{L}_{t}(x_t, \lambda(t))+A\left(x(t)-x_t+\theta t^q\dot{x}(t)\right),
		\end{eqnarray*}
we have 
		\begin{eqnarray*}
			&&-\theta t^{q+s}\langle x(t)-x_t+\theta t^q\dot{x}(t), \nabla_{x}\mathcal{L}(x(t), \lambda(t))+\frac{c}{t^p}x(t)\rangle\\
			&&\quad=-\theta t^{q+s}\langle x(t)-x_t+\theta t^q\dot{x}(t), \nabla_{x}\mathcal{L}_{t}(x(t), \lambda_t)+A^T(\lambda(t)-\lambda_t)\rangle\\
			&&\quad=-\theta t^{q+s}\langle x(t)-x_t, \nabla_{x}\mathcal{L}_{t}(x(t), \lambda_t)\rangle-\theta^2t^{2q+s}\langle\dot{x}(t), \nabla_{x}\mathcal{L}_{t}(x(t), \lambda_t)\rangle\\
			&&\qquad-\theta t^{q+s}\langle x(t)-x_t+\theta t^q\dot{x}(t), A^T(\lambda(t)-\lambda_t)\rangle
		\end{eqnarray*}
		and 
		\begin{eqnarray*}
			&&\theta t^{q+s}\left\langle \lambda(t)-\lambda_t, \nabla_{\lambda}\mathcal{L}(x(t)+\theta t^q\dot{x}(t), \lambda(t))-\frac{c}{t^p}\lambda(t)\right\rangle\\
			&&\quad=\theta t^{q+s}\left\langle \lambda(t)-\lambda_t, \nabla_{\lambda}\mathcal{L}_{t}(x_t, \lambda(t))+A\left(x(t)-x_t+\theta t^q\dot{x}(t)\right)\right\rangle\\
			&&\quad=\theta t^{q+s}\left\langle \lambda(t)-\lambda_t, \nabla_{\lambda}\mathcal{L}_{t}(x_t, \lambda(t))\right\rangle
			+\theta t^{q+s}\left\langle A^T(\lambda(t)-\lambda_t), x(t)-x_t+\theta t^q\dot{x}(t)\right\rangle.
		\end{eqnarray*}
		Since $\mathcal{L}_{t}(\cdot, \lambda_t)$ is $\frac{c}{t^p}$-strongly convex and $-\mathcal{L}_{t}(x_t,\cdot)$ is $\frac{c}{t^p}$-strongly convex, it follows that
		$$\langle x(t)-x_t, \nabla_{x}\mathcal{L}_{t}(x(t), \lambda_t)\rangle\geq\mathcal{L}_{t}(x(t), \lambda_t)-\mathcal{L}_{t}(x_t, \lambda_t)+\frac{c}{2t^p}\|x(t)-x_t\|^2$$
and
		$$-\langle \lambda(t)-\lambda_t, \nabla_{\lambda}\mathcal{L}_{t}(x_t, \lambda(t))\rangle\geq\mathcal{L}_{t}(x_t, \lambda_t)-\mathcal{L}_{t}(x_t, \lambda(t))+\frac{c}{2t^p}\|\lambda(t)-\lambda_t\|^2\geq\frac{c}{2t^p}\|\lambda(t)-\lambda_t\|^2,$$
		where the last equality uses \eqref{zv1}.
		 As a result, we obtain
		\begin{eqnarray}\label{yp1}
			&&-\theta t^{q+s}\langle x(t)-x_t+\theta t^q\dot{x}(t), \nabla_{x}\mathcal{L}(x(t), \lambda(t))+\frac{c}{t^p}x(t)\rangle\nonumber\\
			&&\quad\leq-\theta t^{q+s}\left(\mathcal{L}_{t}(x(t), \lambda_t)-\mathcal{L}_{t}(x_t, \lambda_t)\right)-\frac{c\theta t^{q+s-p}}{2}\|x(t)-x_t\|^2\nonumber\\
			&&\qquad-\theta^2t^{2q+s}\langle\dot{x}(t), \nabla_{x}\mathcal{L}_{t}(x(t), \lambda_t)\rangle
			-\theta t^{q+s}\langle x(t)-x_t+\theta t^q\dot{x}(t), A^T(\lambda(t)-\lambda_t)\rangle
		\end{eqnarray}
		and 
		\begin{eqnarray}\label{yp2}
			&&\theta t^{q+s}\left\langle \lambda(t)-\lambda_t, \nabla_{\lambda}\mathcal{L}(x(t)+\theta t^q\dot{x}(t), \lambda(t))-\frac{c}{t^p}\lambda(t)\right\rangle\nonumber\\
			&&\quad\leq-\frac{c\theta t^{q+s-p}}{2}\|\lambda(t)-\lambda_t\|^2
			+\theta t^{q+s}\left\langle A^T(\lambda(t)-\lambda_t), x(t)-x_t+\theta t^q\dot{x}(t)\right\rangle.
		\end{eqnarray}

Using \eqref{z3}, we have
		\begin{eqnarray*}
			\frac{d}{dt}\mathcal{L}_{t}(x(t),\lambda_t)&=& \langle\nabla f(x(t)), \dot{x}(t)\rangle+\langle Ax(t)-b, \dot{\lambda}_t\rangle+\langle A^T\lambda_t, \dot{x}(t)\rangle+\frac{c}{t^p}\langle x(t), \dot{x}(t)\rangle\\
			&&-\frac{c}{t^p}\langle \lambda_t, \dot{\lambda}_t\rangle-\frac{cp}{2t^{p+1}}\left(\|x(t)\|^2-\|\lambda_t\|^2\right)\\
			&=&\langle\nabla_{x}\mathcal{L}_{t}(x(t),\lambda_t), \dot{x}(t)\rangle+\langle Ax(t)-b, \dot{\lambda}_t\rangle-\frac{c}{t^p}\langle \lambda_t, \dot{\lambda}_t\rangle\\
			&&-\frac{cp}{2t^{p+1}}\left(\|x(t)\|^2-\|\lambda_t\|^2\right)\\
			&=&\langle\nabla_{x}\mathcal{L}_{t}(x(t),\lambda_t), \dot{x}(t)\rangle+\langle Ax(t)-b-\frac{c}{t^p}\lambda_t, \dot{\lambda}_t\rangle-\frac{cp}{2t^{p+1}}\left(\|x(t)\|^2-\|\lambda_t\|^2\right)\\
			&=&\langle\nabla_{x}\mathcal{L}_{t}(x(t),\lambda_t), \dot{x}(t)\rangle+\langle A(x(t)-x_t), \dot{\lambda}_t\rangle-\frac{cp}{2t^{p+1}}\left(\|x(t)\|^2-\|\lambda_t\|^2\right),
		\end{eqnarray*} 
		where the last equality uses \eqref{z4}. This together with  Lemma \ref{lemma012}  yields
		\begin{eqnarray}\label{abfb}
			\frac{d}{dt}\left(\mathcal{L}_{t}(x(t),\lambda_t)-\mathcal{L}_{t}(x_t,\lambda_t)\right)
			&=&\langle\nabla_{x}\mathcal{L}_{t}(x(t),\lambda_t), \dot{x}(t)\rangle+\langle A(x(t)-x_t), \dot{\lambda}_t\rangle\nonumber\\
			&&+\frac{cp}{2t^{p+1}}\left(\|x_t\|^2-\|x(t)\|^2\right).
		\end{eqnarray} 
		Substituting \eqref{yp1}, \eqref{yp2} and \eqref{abfb} into \eqref{z6}, we have 
				\begin{eqnarray}\label{yp3}
			\dot{\mathcal{E}}(t)&\leq&\theta t^{q+s}\left(\theta(2q+s)t^{q-1}-1\right) \left(\mathcal{L}_{t}(x(t),\lambda_t)-\mathcal{L}_{t}(x_t,\lambda_t)\right)+\theta^2t^{2q+s}\langle A(x(t)-x_t), \dot{\lambda}_t\rangle\nonumber\\
			&&+\frac{cp\theta^2}{2}t^{2q+s-p-1}\left(\|x_t\|^2-\|x(t)\|^2\right)+(1-\alpha\theta+\theta qt^{q-1})\theta t^q\|\dot{x}(t)\|^2\nonumber\\
			&&-(\alpha\theta-\theta q t^{q-1})\langle x(t)-x_t, \dot{x}_t\rangle-\theta t^q\langle\dot{x}(t), \dot{x}_t\rangle-\theta\langle \lambda(t)-\lambda_t, \dot{\lambda}_t \rangle\nonumber\\
			&&+\frac{1}{2}\left(q(1-q)\theta t^{q-2}-c\theta t^{q+s-p}\right)\|x(t)-x_t\|^2-\frac{c\theta t^{q+s-p}}{2}\|\lambda(t)-\lambda_t\|^2.
		\end{eqnarray}
Notice that there exists $t_1\ge t_0$ such that
\begin{eqnarray}\label{zg2}
		\alpha\theta-1-\theta q t^{q-1}>0
		\end{eqnarray}
for all $t\ge t_1$ (since  $\theta>\frac{1}{\alpha}$ and $0\leq q<1$).
Because
		\begin{eqnarray*}
			\theta^2 t^{2q+s}\langle A(x(t)-x_t), \dot{\lambda}_t\rangle\leq\frac{\theta^2\|A\|^2t^{2q+s-1}}{2}\|x(t)-x_t\|^2+\frac{\theta^2 t^{2q+s+1}}{2}\|\dot{\lambda}_t\|^2,
		\end{eqnarray*}		
		$$-\langle \dot{x}(t), \dot{x}_t\rangle\leq\frac{1}{2a_1}\|\dot{x}(t)\|^2+\frac{a_1}{2}\|\dot{x}_t\|^2,$$
		$$-\langle x(t)-x_t, \dot{x}_t\rangle\leq\frac{t^{q+s-p}}{2a_2}\|x(t)-x_t\|^2+\frac{a_2t^{p-q-s}}{2}\|\dot{x}_t\|^2,$$
and
$$-\langle \lambda(t)-\lambda_t, \dot{\lambda}_t \rangle\leq\frac{t^{q+s-p}}{2a_3}\|\lambda(t)-\lambda_t\|^2+\frac{a_3t^{p-q-s}}{2}\|\dot{\lambda}_t\|^2$$ where  $a_1>0$, $a_2>0$ and $a_3>0$ are arbitrary constants, it follows from \eqref{yp3} and \eqref{zg2}  that
		\begin{eqnarray}\label{ypz1}
			\dot{\mathcal{E}}(t)&\leq&\theta t^{q+s}\left(\theta(2q+s)t^{q-1}-1\right) \left(\mathcal{L}_{t}(x(t),\lambda_t)-\mathcal{L}_{t}(x_t,\lambda_t)\right)\nonumber\\
			&&+\frac{1}{2}\left(\theta^2\|A\|^2t^{2q+s-1}+q(1-q)\theta t^{q-2}-c\theta t^{q+s-p}+\frac{(\alpha\theta-\theta qt^{q-1})t^{q+s-p}}{a_2}\right)\|x(t)-x_t\|^2\nonumber\\
			&&+\theta t^q\left(1-\alpha\theta+\frac{1}{2a_1}+\theta qt^{q-1}\right)\|\dot{x}(t)\|^2+\frac{\theta}{2}\left(\frac{1}{a_3}-c\right)t^{q+s-p}\|\lambda(t)-\lambda_t\|^2\nonumber\\
			&&+\frac{\theta}{2}\left(a_1 t^q+(\alpha-qt^{q-1})a_2t^{p-q-s}\right)\|\dot{x}_t\|^2+\frac{\theta}{2}\left(\theta t^{2q+s+1}+a_3t^{p-q-s}\right)\|\dot{\lambda}_t\|^2\nonumber\\
			&&+\frac{cp\theta^2}{2}t^{2q+s-p-1}\left(\|x_t\|^2-\|x(t)\|^2\right) 
		\end{eqnarray}
for all $t\ge t_1$. Again using \eqref{z8}, we have
		\begin{eqnarray*}
			\mathcal{E}(t)&\leq&\theta^2 t^{2q+s}\left(\mathcal{L}_{t}(x(t),\lambda_t)-\mathcal{L}_{t}(x_t,\lambda_t)\right)+\|x(t)-x_t\|^2+\theta^2t^{2q}\|\dot{x}(t)\|^2\nonumber\\
			&&+\frac{\alpha\theta-1-\theta q t^{q-1}}{2}\|x(t)-x_t\|^2+\frac{\theta}{2}\|\lambda(t)-\lambda_t\|^2\\
			&=&\theta^2 t^{2q+s}\left(\mathcal{L}_{t}(x(t),\lambda_t)-\mathcal{L}_{t}(x_t,\lambda_t)\right)+\frac{\alpha\theta+1-\theta q t^{q-1}}{2}\|x(t)-x_t\|^2\nonumber\\
			&&+\theta^2t^{2q}\|\dot{x}(t)\|^2+\frac{\theta}{2}\|\lambda(t)-\lambda_t\|^2.
		\end{eqnarray*}
This together with \eqref{ypz1} yields the desired result.
	\end{proof}

Next, we apply Lemma \ref{lemma3.1} to establish the strong convergence of the trajectory of \eqref{z2}  and the convergence rates of the primal-dual gap, the objective residual and the feasibility violation.
	
\begin{theorem}\label{ztt3.1}
	Suppose that Assumption \ref{AS-F} holds and  $p-q-1<s<1-3q$. Let $(x(t),\lambda(t))$ be a solution of \eqref{z2} and $(x^*,\lambda^*)$ be the minimum norm element of  $\Omega$. Then, 
$$\lim_{t\rightarrow+\infty}\|(x(t), \lambda(t))-(x^*,\lambda^*)\|=0$$ and  the following conclusions hold:
		\begin{itemize}
			\item[(i)] When $p-q-1<s<\frac{p-3q-1}{2}$, it holds
			$$\|x(t)-x_t\|^2\leq\mathcal{O}\left(\frac{1}{t^{2(1+s+q-p)}}\right), \quad \|\lambda(t)-\lambda_t\|^2\leq\mathcal{O}\left(\frac{1}{t^{2(1+s+q-p)}}\right),$$ $$\|\dot{x}(t)\|^2\leq\mathcal{O}\left(\frac{1}{t^{2(1+s+2q-p)}}\right),\quad 	\|Ax(t)-b\|\leq\mathcal{O}\left(\frac{1}{t^{p}}+\frac{1}{t^{1+s+q-p}}\right).$$
			Further, if $\frac{2p-2-4q}{3}<s<\frac{p-3q-1}{2}$, then it also holds
			$$\mathcal{L}_{t}(x(t),\lambda_t)-\mathcal{L}_{t}(x_t,\lambda_t)\leq\mathcal{O}\left(\frac{1}{t^{4q+3s-2p+2}}\right).$$
			$$\mathcal{L}(x(t),\lambda^*)-\mathcal{L}(x^*,\lambda^*)\leq\mathcal{O}\left(\frac{1}{t^{p}}+\frac{1}{t^{4q+3s-2p+2}}\right),$$ $$|f(x(t))-f(x^*)|\leq\mathcal{O}\left(\frac{1}{t^{p}}+\frac{1}{t^{4q+3s-2p+2}}\right).$$

			\item[(ii)] When $\frac{p-3q-1}{2}\leq s<1-3q$, it holds
			$$\mathcal{L}_{t}(x(t),\lambda_t)-\mathcal{L}_{t}(x_t,\lambda_t)\leq\mathcal{O}\left(\frac{1}{t^{1-r}}\right), \quad \|x(t)-x_t\|^2\leq\mathcal{O}\left(\frac{1}{t^{1-2q-s-r}}\right),$$ $$\|\lambda(t)-\lambda_t\|^2\leq\mathcal{O}\left(\frac{1}{t^{1-2q-s-r}}\right), \quad \|\dot{x}(t)\|^2\leq\mathcal{O}\left(\frac{1}{t^{1-s-r}}\right),$$
			$$\mathcal{L}(x(t),\lambda^*)-\mathcal{L}(x^*,\lambda^*)\leq\mathcal{O}\left(\frac{1}{t^{p}}+\frac{1}{t^{(1-2q-s-r)/2}}\right),$$
			$$|f(x(t))-f(x^*)|\leq\mathcal{O}\left(\frac{1}{t^{p}}+\frac{1}{t^{(1-2q-s-r)/2}}\right), \quad \|Ax(t)-b\|\leq\mathcal{O}\left(\frac{1}{t^{p}}+\frac{1}{t^{(1-2q-s-r)/2}}\right),$$
			where $r=\max\{q, p-q-s\}$.
			
		\end{itemize}
\end{theorem} 
\begin{proof}
According to Assumption \ref{AS-F}, we can take $a_1, a_2, a_3, r$, and $K$ in Lemma \ref{lemma3.1} such that $r=\max\{q, p-q-s\}$, $a_1>\frac{1}{2(\alpha\theta-1)}$, $a_2>\frac{\alpha}{c}$,  $a_3>\frac{1}{c}$, and
		$$0<K<\min\left\{\frac{1}{\theta}, \alpha-\frac{1}{\theta}\left(1+\frac{1}{2a_1}\right), \frac{ca_2-\alpha}{a_2(\alpha+\frac{1}{\theta})}, c-\frac{1}{a_3}\right\}.$$
		By  Lemma \ref{lemma3.1} and using \eqref{zv1}, there exists a constant $t_2\geq\max\{t_1,1\}$ such that 
		\begin{eqnarray*}
			\dot{\mathcal{E}}(t)+\frac{K}{t^r}\mathcal{E}(t)
			&\leq&\frac{\theta}{2}\left(a_1t^q+(\alpha- qt^{q-1})a_2t^{p-q-s}\right)\|\dot{x}_t\|^2+\frac{\theta}{2}\left(\theta t^{2q+s+1}+a_3 t^{p-q-s}\right)\|\dot{\lambda}_t\|^2\\
			&&+\frac{cp\theta^2}{2}t^{2q+s-p-1}\|x_t\|^2, \quad \forall t\geq t_2.
		\end{eqnarray*}	
		Denote $y_t=(x_t, \lambda_t)$ and $y^*=(x^*,\lambda^*)$. By Lemma \ref{lemma2.1},
		$$\|x_t\|^2\leq\|y_t\|^2\leq\|y^*\|^2$$
and
		$$\max\{\|\dot{x}_t\|^2,\|\dot{\lambda}_t\|^2\}\leq\|\dot{y}_t\|^2\leq\frac{p^2}{t^2}\|y_t\|^2\leq\frac{p^2}{t^2}\|y^*\|^2 $$
 for $t\geq t_0$.
		As a result, we get  for any $t\geq t_2$,
		\begin{eqnarray}\label{z11}
			\dot{\mathcal{E}}(t)+\frac{K}{t^r}\mathcal{E}(t)
			&\leq&\frac{\theta p^2}{2}\left(a_1t^{q-2}+(\alpha- qt^{q-1})a_2t^{p-q-s-2}\right)\|y^*\|^2\nonumber\\
			&&+\frac{\theta p^2}{2}\left(\theta t^{2q+s-1}+a_3 t^{p-q-s-2}\right)\|y^*\|^2+\frac{cp\theta^2}{2}t^{2q+s-p-1}\|y^*\|^2\nonumber\\
			&\leq&\frac{\theta p\|y^*\|^2}{2}\left(a_1pt^{q-2}+(\alpha a_2+a_3)pt^{p-q-s-2}+\theta pt^{2q+s-1}+c\theta t^{2q+s-p-1}\right)\nonumber\\
			&\leq&\frac{\theta p\|y^*\|^2}{2}\left(a_1pt^{q-2}+(\alpha a_2+a_3)pt^{p-q-s-2}+(p+c)\theta t^{2q+s-1}\right),
		\end{eqnarray}
		where the last inequality uses $2q+s-p-1\leq2q+s-1$.  
		
	It follows from \eqref{z4} and Lemma \ref{lemma2.1} that
		\begin{eqnarray}\label{zd1}
			\|Ax(t)-b\|&=&\|A(x(t)-x_t)+Ax_t-b\|\nonumber\\
			&\leq&\|A\|\cdot\|x(t)-x_t\|+\|Ax_t-b\|\nonumber\\
			&=&\|A\|\cdot\|x(t)-x_t\|+\frac{c}{t^p}\|\lambda_t\|\nonumber\\
			&\leq&\|A\|\cdot\|x(t)-x_t\|+\frac{c}{t^p}\|y^*\|.
		\end{eqnarray}	
		Since $\mathcal{L}_{t}(x_t,\lambda_t)\leq\mathcal{L}_{t}(x^*,\lambda_t)$ and  $(x^*,\lambda^*)\in\Omega$, it follows from \eqref{zc1} and \eqref{z3} that
		\begin{eqnarray*}
			\mathcal{L}_{t}(x(t),\lambda_t)-\mathcal{L}_{t}(x_t,\lambda_t)&\geq&\mathcal{L}_{t}(x(t),\lambda_t)-\mathcal{L}_{t}(x^*,\lambda_t)\\
			&=&\mathcal{L}(x(t),\lambda_t)-\mathcal{L}(x^*,\lambda_t)+\frac{c}{2t^p}\left(\|x(t)\|^2-\|x^*\|^2\right)\\
			&=&\mathcal{L}(x(t),\lambda_t)-\mathcal{L}(x^*,\lambda^*)+\frac{c}{2t^p}\left(\|x(t)\|^2-\|x^*\|^2\right)\\
			&=&\mathcal{L}(x(t),\lambda^*)-\mathcal{L}(x^*,\lambda^*)+\langle\lambda_t-\lambda^*, Ax(t)-b\rangle\\
			&&+\frac{c}{2t^p}\left(\|x(t)\|^2-\|x^*\|^2\right)\\
			&\geq&\mathcal{L}(x(t),\lambda^*)-\mathcal{L}(x^*,\lambda^*)-\|\lambda_t-\lambda^*\|\cdot\| Ax(t)-b\|\\
			&&+\frac{c}{2t^p}\left(\|x(t)\|^2-\|x^*\|^2\right).
		\end{eqnarray*}	
		This implies
		\begin{eqnarray}\label{z15}
			0\leq\mathcal{L}(x(t),\lambda^*)-\mathcal{L}(x^*,\lambda^*)
			&\leq&\mathcal{L}_{t}(x(t),\lambda_t)-\mathcal{L}_{t}(x_t,\lambda_t)+\|\lambda_t-\lambda^*\|\cdot\| Ax(t)-b\|\nonumber\\
			&&+\frac{c}{2t^p}\left(\|x^*\|^2-\|x(t)\|^2\right).
		\end{eqnarray}	
		By the definition of  $\mathcal{L}$, we obtain 
		\begin{eqnarray}\label{z16}
			|f(x(t))-f(x^*)|\leq\mathcal{L}(x(t),\lambda^*)-\mathcal{L}(x^*,\lambda^*)+\|\lambda^*\|\cdot\|Ax(t)-b\|.
		\end{eqnarray}
		
Next, we analyze separately  cases (i) and  (ii).

		{\bf (i)} If $p-q-1<s<\frac{p-3q-1}{2}$, then
		\begin{eqnarray*}
			\begin{cases}r=\max\{q, p-q-s\}=p-q-s\in (0,1),\\
				q-2<2q+s-1,\\
				2q+s-1<  p-q-s-2.
			\end{cases}
		\end{eqnarray*} 
As a result, there exist $t_2\geq\max\{t_1, 1\}$  and $C_1>0$ such that
		$$\frac{\theta p\|y^*\|^2}{2}\left(a_1pt^{q-2}+(\alpha a_2+a_3)pt^{p-q-s-2}+(p+c)\theta t^{2q+s-1}\right)\leq C_1t^{p-q-s-2}, \quad \forall t\geq t_2.$$
		This together with \eqref{z11} yields
		\begin{eqnarray}\label{z9}
			\dot{\mathcal{E}}(t)+\frac{K}{t^r}\mathcal{E}(t)
			&\leq&C_1 t^{p-q-s-2}, \quad \forall t\geq t_2.
		\end{eqnarray}	
		Multiplying both sides of \eqref{z9} by $e^{\frac{K}{1-r}t^{1-r}}$, we have 
		\begin{eqnarray}\label{4yp1}
			\frac{d}{dt}\left(e^{\frac{K}{1-r}t^{1-r}}\mathcal{E}(t)\right)
			&\leq&C_1 t^{p-q-s-2}e^{\frac{K}{1-r}t^{1-r}}, \quad \forall t\geq t_2.
		\end{eqnarray}	
		Since $r=p-q-s$, $0<r<1$ and $s>p-q-1$, we get
		\begin{eqnarray}\label{z10}
			\begin{cases}
				p-q-s-2+r=2(p-q-s-1)<0,\\
				p-q-s-3+r=p-q-s-2+r-1< p-q-s-2.
			\end{cases}
		\end{eqnarray} 
		Then, there exist  $C_2\in(0,1)$ and $t_3\geq t_2$ such that
		$$(1-C_2)Kt^{p-q-s-2}+(p-q-s-2+r)t^{p-q-s-3+r}\geq0, \quad \forall t\geq t_3.$$
	It follows that
		\begin{eqnarray*}
			\frac{d}{dt}\left(t^{p-q-s-2+r}e^{\frac{K}{1-r}t^{1-r}}\right)&=&\left((p-q-s-2+r)t^{p-q-s-3+r}+(1-C_2)Kt^{p-q-s-2}\right)e^{\frac{K}{1-r}t^{1-r}}\\
			&&+C_2Kt^{p-q-s-2}e^{\frac{K}{1-r}t^{1-r}}\\
			&\geq& C_2Kt^{p-q-s-2}e^{\frac{K}{1-r}t^{1-r}}, \quad \forall t\geq t_3.
		\end{eqnarray*}	
	This together with  \eqref{4yp1} yields
		\begin{eqnarray*}
			\frac{d}{dt}\left(e^{\frac{K}{1-r}t^{1-r}}\mathcal{E}(t)\right)
			&\leq&\frac{C_1}{C_2K}\frac{d}{dt}\left(t^{p-q-s-2+r}e^{\frac{K}{1-r}t^{1-r}}\right), \quad \forall t\geq t_3,
		\end{eqnarray*}	
		which implies that for every $t\geq t_3$ 
		\begin{eqnarray*}
			e^{\frac{K}{1-r}t^{1-r}}\mathcal{E}(t)
			&\leq&e^{\frac{K}{1-r}t_{3}^{1-r}}\mathcal{E}(t_3)+
			\frac{C_1}{C_2K}\left(t^{p-q-s-2+r}e^{\frac{K}{1-r}t^{1-r}}-t_{3}^{p-q-s-2+r}e^{\frac{K}{1-r}t_{3}^{1-r}}\right).
		\end{eqnarray*}	
		As a result,
		\begin{eqnarray*}
			\mathcal{E}(t)
			&\leq&\frac{C_3}{e^{\frac{K}{1-r}t^{1-r}}}+
			\frac{C_1}{C_2K}t^{p-q-s-2+r}, \quad \forall t\geq t_3,
		\end{eqnarray*}	
		where $C_3=e^{\frac{K}{1-r}t_{3}^{1-r}}\mathcal{E}(t_3)-\frac{C_1}{C_2K}t_{3}^{p-q-s-2+r}e^{\frac{K}{1-r}t_{3}^{1-r}}$. This, combined with  $r=p-q-s$, implies that 
		\begin{eqnarray}\label{z12}
			\mathcal{E}(t)
			\leq C_4t^{2(p-q-s-1)}
			, \quad \forall t\geq t_4,
		\end{eqnarray}	
		where  $t_4\ge t_3$ and $C_4>0$ is a constant.
		It follows from \eqref{z8}, \eqref{z10} and \eqref{z12} that
		\begin{eqnarray}\label{zj1}
		\|x(t)-x_t\|^2\leq\mathcal{O}\left(\frac{1}{t^{2(1+s+q-p)}}\right), \quad \|\lambda(t)-\lambda_t\|^2\leq\mathcal{O}\left(\frac{1}{t^{2(1+s+q-p)}}\right),
    	\end{eqnarray}	
		$$\|\dot{x}(t)\|^2\leq\mathcal{O}\left(\frac{1}{t^{2(1+s+2q-p)}}\right).$$
		Using \eqref{zj1}  and \eqref{zd1}, we obtain 
		\begin{eqnarray*}\label{zdf1} 
			\|Ax(t)-b\|\leq\mathcal{O}\left(\frac{1}{t^{p}}+\frac{1}{t^{1+s+q-p}}\right).
		\end{eqnarray*}
	
		Further, if $\frac{2p-2-4q}{3}<s<\frac{p-3q-1}{2}$,  from \eqref{z8} and \eqref{z12} we have 
		\begin{equation}\label{4ypa}
\mathcal{L}_{t}(x(t),\lambda_t)-\mathcal{L}_{t}(x_t,\lambda_t)\leq\mathcal{O}\left(\frac{1}{t^{4q+3s-2p+2}}\right).
\end{equation}
		 Since $x_t\to x^*$ as $t\to +\infty$ and $s>\frac{2p-2-4q}{3}>p-q-1$, from \eqref{zj1} and Lemma \ref{lemma2.1} we have  $\lim_{t\rightarrow+\infty}\|x^*\|^2-\|x(t)\|^2=0$.
Using \eqref{4ypa}, \eqref{z15} and  the fact $4q+3s-2p+2<1+s+q-p$, we obtain
		$$\mathcal{L}(x(t),\lambda^*)-\mathcal{L}(x^*,\lambda^*)\leq\mathcal{O}\left(\frac{1}{t^{p}}+\frac{1}{t^{4q+3s-2p+2}}\right).$$ 
		This together with \eqref{z16} implies 
		$$|f(x(t))-f(x^*)|\leq\mathcal{O}\left(\frac{1}{t^{p}}+\frac{1}{t^{4q+3s-2p+2}}\right).$$
		
		{\bf (ii)} If $\frac{p-3q-1}{2}\leq s<1-3q$, then
		\begin{eqnarray*}
			\begin{cases}
				q-2<2q+s-1,\\
				p-q-s-2\leq2q+s-1.
			\end{cases}
		\end{eqnarray*} 
		Consequently, there exists $\widehat{C}_1>0$ such that
		$$\frac{\theta p\|y^*\|^2}{2}\left(a_1pt^{q-2}+(\alpha a_2+a_3)pt^{p-q-s-2}+(p+c)\theta t^{2q+s-1}\right)\leq \widehat{C}_1t^{2q+s-1}, \quad \forall t\geq t_2.$$
		This combined with \eqref{z11} yields 
		\begin{eqnarray*}
			\dot{\mathcal{E}}(t)+\frac{K}{t^r}\mathcal{E}(t)
			&\leq&\widehat{C}_1 t^{2q+s-1}, \quad \forall t\geq t_2,
		\end{eqnarray*}	
		where $r=\max\{q, p-q-s\}\in[0, 1)$.
		By   similar arguments as in proving \eqref{z12}, we  have 
		\begin{eqnarray*}
			\mathcal{E}(t)
			\leq\widehat{C}_4t^{2q+s-1+r}, \quad \forall t\geq \hat{t}_4,
		\end{eqnarray*}	
where $\widehat{C}_4>0$ is a constant and $\hat{t}_4\geq t_2$.
This together with \eqref{z8} implies 
		\begin{eqnarray}\label{fz13}
			\mathcal{L}_{t}(x(t),\lambda_t)-\mathcal{L}_{t}(x_t,\lambda_t)\leq\mathcal{O}\left(\frac{1}{t^{1-r}}\right),
		\end{eqnarray}	
		\begin{eqnarray}\label{zj2}
		\|x(t)-x_t\|^2\leq\mathcal{O}\left(\frac{1}{t^{1-2q-s-r}}\right), \quad\|\lambda(t)-\lambda_t\|^2\leq\mathcal{O}\left(\frac{1}{t^{1-2q-s-r}}\right),
    	\end{eqnarray}	
		$$\|\dot{x}(t)\|^2\leq\mathcal{O}\left(\frac{1}{t^{1-s-r}}\right).$$
By Lemma \ref{lemma2.1} and using  \eqref{zj2}, we have  
\begin{equation}\label{zhfa2}
\lim_{t\rightarrow+\infty}\|x^*\|^2-\|x(t)\|^2=0.
\end{equation}

		Using \eqref{zd1} and \eqref{zj2}, we  obtain 
		\begin{eqnarray}\label{zhfa1}
			\|Ax(t)-b\|\leq\mathcal{O}\left(\frac{1}{t^{p}}+\frac{1}{t^{(1-2q-s-r)/2}}\right).
		\end{eqnarray}
Since $r=\max\{q, p-q-s\}$, $\frac{p-3q-1}{2}\leq s<1-3q$ and $0<p<1-q$, we have
       \begin{eqnarray*}
       	       		\min\{(1-2q-s-r)/2,p\} \leq1-r.       	   
       \end{eqnarray*} 
It follows from  \eqref{z15}, \eqref{fz13}, \eqref{zhfa2}, and \eqref{zhfa1} that

$$\mathcal{L}(x(t),\lambda^*)-\mathcal{L}(x^*,\lambda^*)\leq\mathcal{O}\left(\frac{1}{t^{p}}+\frac{1}{t^{(1-2q-s-r)/2}}\right).$$     
Using \eqref{z16} again, we have
$$|f(x(t))-f(x^*)|\leq\mathcal{O}\left(\frac{1}{t^{p}}+\frac{1}{t^{(1-2q-s-r)/2}}\right).$$

	Summarizing (i) and (ii), we have
	$$\lim_{t\rightarrow+\infty}\|(x(t), \lambda(t))-(x^*,\lambda^*)\|=0.$$
	\end{proof}
	
When  $p-2q<s<1-3q$, we can improve the convergence rates obtained in  $(ii)$ of Theorem \ref{ztt3.1}.
	\begin{theorem}\label{ztt3.3}
		Suppose that Assumption \ref{AS-F} holds and  $p-2q<s<1-3q$. Let $(x(t),\lambda(t))$ be a solution of \eqref{z2} and $(x^*,\lambda^*)$ be the minimum norm element of  $\Omega$. Then,  the following conclusions hold:
		$$\|x(t)-x_t\|^2\leq\mathcal{O}\left(\frac{1}{t^{1-(p+q)}}\right),\quad\mathcal{L}(x(t),\lambda^*)-\mathcal{L}(x^*,\lambda^*)\leq\mathcal{O}\left(\frac{1}{t^{p}}+\frac{1}{t^{(1-(p+q))/2}}\right),$$
		$$|f(x(t))-f(x^*)|\leq\mathcal{O}\left(\frac{1}{t^{p}}+\frac{1}{t^{(1-(p+q))/2}}\right), \quad \|Ax(t)-b\|\leq\mathcal{O}\left(\frac{1}{t^{p}}+\frac{1}{t^{(1-(p+q))/2}}\right).$$
	\end{theorem} 
	\begin{proof}
 By Assumption \ref{AS-F} and  $p-2q<s<1-3q$, we have $\frac{p-3q-1}{2}<s<1-3q$ and $r=\max\{q, p-q-s\}=q$. By $(ii)$ of Theorem \ref{ztt3.1},  
	\begin{equation}\label{zhfb1}
\mathcal{L}_{t}(x(t),\lambda_t)-\mathcal{L}_{t}(x_t,\lambda_t)\leq\mathcal{O}\left(\frac{1}{t^{1-q}}\right).
\end{equation}
	Since $\mathcal{L}_{t}(\cdot, \lambda_t)$ is $\frac{c}{t^p}$-strongly convex, it follows from \eqref{z4} that
	\begin{eqnarray*}
		\mathcal{L}_{t}(x(t),\lambda_t)-\mathcal{L}_{t}(x_t,\lambda_t)\geq\frac{c}{2t^p}\|x(t)-x_t\|^2.
	\end{eqnarray*}	
	This together with \eqref{zhfb1} yields 
	\begin{eqnarray}\label{zj3}
		\|x(t)-x_t\|^2\leq\mathcal{O}\left(\frac{1}{t^{1-(p+q)}}\right).
	\end{eqnarray}	
	Combining \eqref{zj3} and \eqref{zd1}, we get
	$$\|Ax(t)-b\|\leq\mathcal{O}\left(\frac{1}{t^p}+\frac{1}{t^{(1-(q+p))/2}}\right).$$

By similar arguments as in Theorem  Theorem \ref{ztt3.1}, we have
	$$\mathcal{L}(x(t),\lambda^*)-\mathcal{L}(x^*,\lambda^*)\leq\mathcal{O}\left(\frac{1}{t^p}+\frac{1}{t^{(1-(q+p))/2}}\right)$$
	and
	$$|f(x(t))-f(x^*)|\leq\mathcal{O}\left(\frac{1}{t^p}+\frac{1}{t^{(1-(q+p))/2}}\right).$$
	\end{proof}
	
	\begin{remark}\label{remark3}

When  Assumption \ref{AS-F} holds and  $p-2q<s<1-3q$,  it is easy to verify that 
$$1-(p+q)>1-2q-s-r.$$
Therefore,   Theorem \ref{ztt3.3}  improves  $(ii)$ of Theorem \ref{ztt3.1} when $p-2q<s<1-3q$.
	\end{remark}

	In Theorem \ref{ztt3.1}, we not only prove  the strong convergence of the trajectory of \eqref{z2} to the minimal norm element of $\Omega$, but also  establish the convergence rates of the primal-dual gap, the objective residual, the feasibility violation. Next, by  using the approaches  in  \cite{HeHFiietal(2022), HeTLFconver(2023),ChbaniRBOn(2024)}, we can improve these rates  under  suitable choices of the parameters $q$, $p$ and $s$. Before doing this, we first give  a lemma.
	
	\begin{lemma}\label{lemma3.2}
		Assume that $0\leq q<1$, $0<p<1-q$ and $p-q-1<s$. Let $(x,\lambda): [t_{0},+\infty)\rightarrow \mathcal{X}\times\mathcal{Y}$ be a solution of \eqref{z2}. If  $(\lambda(t))_{t\geq t_0}$ is bounded, then for every $T\geq t_0>0$ there exists a constant $\widetilde{C}_{T}\geq0$ such that
		\begin{eqnarray*}
			\left\|\theta t^{2q+s}(Ax(t)-b)+\frac{1}{h(t)}\int_{T}^{t}\theta \tau^{2q+s}(Ax(\tau)-b)V(\tau)d\tau\right\|\leq\widetilde{C}_{T}, \quad \forall t\geq T,
		\end{eqnarray*}	
		where $V(t)=\left(\frac{t^{-q}}{\theta}-(2q+s)t^{-1}-ct^{-p+q+s}\right)h(t)$ and $h(t)=e^{\frac{c}{1-(p-q-s)}t^{1-(p-q-s)}}$.
	\end{lemma}
	
	\begin{proof}
		Using \eqref{z2} and \eqref{z3}, we have 	
		\begin{eqnarray*}
			\dot{\lambda}(t)+ct^{q-p+s}\lambda(t)=t^{q+s}(Ax(t)-b)+\theta t^{2q+s}A\dot{x}(t), \quad \forall t\geq t_0.
		\end{eqnarray*}	
		Let $h(t)=e^{\frac{c}{1-(p-q-s)}t^{1-(p-q-s)}}$. Then,
		\begin{eqnarray*}
			\frac{d}{dt}\left(h(t)\lambda(t)\right)=h(t)t^{q+s}(Ax(t)-b) +h(t)\theta t^{2q+s}A\dot{x}(t), \quad \forall t\geq t_0.
		\end{eqnarray*}	
		Given $T\geq t_0>0$, integrating the above equality from $T$ to $t$ gives
		\begin{eqnarray*}
			h(t)\lambda(t)&=&h(T)\lambda(T)+\int_{T}^{t}h(\tau)\tau^{q+s}(Ax(\tau)-b)d\tau+\int_{T}^{t}h(\tau)\theta \tau^{2q+s}d(Ax(\tau)-b)\\
			&=&h(T)\lambda(T)+\int_{T}^{t}h(\tau)\tau^{q+s}(Ax(\tau)-b)d\tau+h(t)\theta t^{2q+s}(Ax(t)-b)\\
			&&-h(T)\theta T^{2q+s}(Ax(T)-b)-\int_{T}^{t}\theta (Ax(\tau)-b)(2q+s)\tau^{2q+s-1}h(\tau)d\tau\\
			&&-\int_{T}^{t}\theta (Ax(\tau)-b)\tau^{2q+s}\dot{h}(\tau)d\tau\\
			&=&h(T)\lambda(T)-h(T)\theta T^{2q+s}(Ax(T)-b)+h(t)\theta t^{2q+s}(Ax(t)-b)\\
			&&+\int_{T}^{t}\theta \tau^{2q+s}(Ax(\tau)-b)\left(\frac{\tau^{-q}}{\theta}-(2q+s)\tau^{-1}-c\tau^{q+s-p}\right)h(\tau)d\tau, \quad \forall t\geq T,
		\end{eqnarray*}	
		where the last equality uses $\dot{h}(\tau)=ct^{q+s-p}h(\tau)$. This yields
		\begin{eqnarray*}
			\lambda(t)
			&=&\frac{h(T)\lambda(T)-h(T)\theta T^{2q+s}(Ax(T)-b)}{h(t)}+\theta t^{2q+s}(Ax(t)-b)\\
			&&+\frac{1}{h(t)}\int_{T}^{t}\theta \tau^{2q+s}(Ax(\tau)-b)\left(\frac{\tau^{-q}}{\theta}-(2q+s)\tau^{-1}-c\tau^{q+s-p}\right)h(\tau)d\tau, \quad \forall t\geq T.
		\end{eqnarray*}	
		Let $V(t)=\left(\frac{t^{-q}}{\theta}-(2q+s)t^{-1}-ct^{q+s-p}\right) h(t)$. Since $\lim_{t\rightarrow+\infty}h(t)=+\infty$ and $(\lambda(t))_{t\geq t_0}$ is bounded, there exists a constant $\widetilde{C}_{T}>0$ such that
		\begin{eqnarray*}
			\left\|\theta t^{2q+s}(Ax(t)-b)+\frac{1}{h(t)}\int_{T}^{t}\theta \tau^{2q+s}(Ax(\tau)-b)V(\tau)d\tau\right\|\leq\widetilde{C}_{T}, \quad \forall t\geq T.
		\end{eqnarray*}	
	\end{proof}

Next, we apply Lemma  \ref{lemma2.2} and Lemma  \ref{lemma2.3} to establish the convergence rate   $\mathcal{O}\left(\frac{1}{t^{2q+s}}\right)$ of the primal-dual gap, the objective residual, and the feasibility violation along the trajectory of  \eqref{z2} when  $-2q<s\leq p-2q$. 
	
	\begin{theorem}\label{ztt3.2}
	Suppose that Assumption \ref{AS-F} holds and $-2q<s\leq p-2q$.
	Let $(x(t),\lambda(t))$ be a solution of \eqref{z2} and $(x^*,\lambda^*)$ be the minimum norm element of $\Omega$. Then, as $t\rightarrow+\infty$ it holds
	$$\mathcal{L}(x(t),\lambda^*)-\mathcal{L}(x^*,\lambda^*)\leq\mathcal{O}\left(\frac{1}{t^{2q+s}}\right),$$
	$$|f(x(t))-f(x^*)|\leq\mathcal{O}\left(\frac{1}{t^{2q+s}}\right), \quad \|Ax(t)-b\|\leq\mathcal{O}\left(\frac{1}{t^{2q+s}}\right).$$
   \end{theorem} 
   \begin{proof}

By  Theorem \ref{ztt3.1}, $\lim_{t\rightarrow+\infty}\|\lambda(t)-\lambda^*\|=0$. Therefore, the conclusion of Lemma  \ref{lemma3.2} holds. Consequently,  for any $T\geq t_0>0$ there exists a constant $\widetilde{C}_{T}\geq0$ such that
		\begin{eqnarray}\label{zhfa3}
			\left\|\theta t^{2q+s}(Ax(t)-b)+\frac{1}{h(t)}\int_{T}^{t}\theta \tau^{2q+s}(Ax(\tau)-b)V(\tau)d\tau\right\|\leq\widetilde{C}_{T}, \quad \forall t\geq T,
		\end{eqnarray}	
		where 
$$h(t)=e^{\frac{c}{1-(p-q-s)}t^{1-(p-q-s)}}$$
and $$V(t)=\left(\frac{t^{-q}}{\theta}-(2q+s)t^{-1}-ct^{-p+q+s}\right)h(t).$$

Next, we will analyze separately the following two situations.
	
	{\bf Case I:} $-2q<s<p-2q$. In this case, we have $-p+q+s<-q$ and $2q+s>0$.  Then, there exists $\tilde{t}_1\geq \max\{t_0,1\}$ such that 
	$$V(t)\geq0, \quad \forall t\geq\tilde{t}_1.$$
Let
 $$\delta=\tilde{t}_1,\quad \mu=\frac{c}{1-(p-q-s)}>0,\quad \nu=1-(p-q-s)>0,$$
$$g(t)=\theta t^{2q+s}(Ax(\tau)-b)\quad\text{ and }\quad a(t)=V(t).$$
Applying Lemma \ref{lemma2.2} to \eqref{zhfa3} with  $T=\tilde{t}_1$,
	we have
	$$\sup_{t\geq \tilde{t}_1}\|\theta t^{2q+s}(Ax(t)-b)\|<+\infty,$$
	which means 
	\begin{eqnarray*}
		\|Ax(t)-b\|\leq\mathcal{O}\left(\frac{1}{t^{2q+s}}\right).
	\end{eqnarray*}	
	 
	{\bf Case II:} $s=p-2q$. In this case, we have $-p+q+s=-q$, and so
	\begin{eqnarray}\label{z17}
		V(t)=\left(\left(\frac{1}{\theta}-c\right)t^{-q}-pt^{-1}\right)h(t),
	\end{eqnarray}	
	where 
$$h(t)=e^{\frac{c}{1-q}t^{1-q}}.$$
According to the sign of $\frac{1}{\theta}-c$, we analyze separately the following two subcases.
	
	{\bf Subcase I:} $\frac{1}{\theta}-c>0$. In this subcase, there exists $\tilde{t}_2\geq \max\{t_0,1\}$ such that 
	$$V(t)\geq0, \quad \forall t\geq \tilde{t}_2.$$ 
	Using again Lemma \ref{lemma2.2} with $\delta=\tilde{t}_2$, $\mu=\frac{c}{1-q}>0$, $\nu=1-q>0$, $g(t)=\theta t^{p}(Ax(t)-b)$ and $a(t)=V(t)$  to \eqref{z17} with  $T=\tilde{t}_2$,
	we get 
	\begin{eqnarray}\label{z18}
		\|Ax(t)-b\|\leq\mathcal{O}\left(\frac{1}{t^{p}}\right)=\mathcal{O}\left(\frac{1}{t^{2q+s}}\right).
	\end{eqnarray}	
	
	{\bf Subcase II:}  $\frac{1}{\theta}-c\leq0$. In this subcase, 
	$$V(t)\leq0,  \quad \forall t\geq t_0.$$ 
	Since  $0\leq q<1$,  there exists a constant $\tilde{t}_3\geq \max\{t_0,1\}$ such that 
	$$\frac{1}{\theta}t^{-q}-pt^{-1}\geq\frac{1}{2\theta}t^{-q}, \quad \forall t\geq\tilde{t}_3.$$
It follows from \eqref{z17} that
	\begin{eqnarray*}
		\frac{1}{h(t)}\int_{\tilde{t}_3}^{t}V(\tau)d\tau&=&\frac{1}{h(t)}\int_{\tilde{t}_3}^{t}\left(\left(\frac{1}{\theta}-c\right)\tau^{-q}-p\tau^{-1}\right)h(\tau)d\tau\\
		&\geq&\frac{1}{h(t)}\left(\frac{1}{2c\theta}\int_{\tilde{t}_3}^{t}c\tau^{-q}h(\tau)d\tau-\int_{\tilde{t}_3}^{t}c\tau^{-q}h(\tau)d\tau\right)\\
		&=&-\left(1-\frac{1}{2c\theta}\right)\frac{1}{h(t)}\int_{\tilde{t}_3}^{t}\dot{h}(\tau)d\tau\\
		&=&-\left(1-\frac{1}{2c\theta}\right)\left(1-\frac{h(\tilde{t}_3)}{h(t)}\right)\\
		&=&-\left(1-\frac{1}{2c\theta}\right)+\left(1-\frac{1}{2c\theta}\right)\frac{h(\tilde{t}_3)}{h(t)}\\
		&\geq&-\left(1-\frac{1}{2c\theta}\right)>-1.
	\end{eqnarray*}	
Let
$$\delta=\tilde{t}_3,\quad \mu=\frac{c}{1-q},\quad \nu=1-q>0, \quad C_0=-\left(1-\frac{1}{2c\theta}\right),$$
$$g(t)=\theta  t^{p}(Ax(t)-b),\quad\text{ and } \quad a(t)=V(t).$$
	Applying Lemma \ref{lemma2.3}  to \eqref{zhfa3} with  $T=\tilde{t}_3$,
	we get
	$$\sup_{t\geq \tilde{t}_3}\|\theta t^{p}(Ax(t)-b)\|<+\infty,$$
	which together with $s=p-2q$ yields 
	\begin{eqnarray*}
		\|Ax(t)-b\|\leq\mathcal{O}\left(\frac{1}{t^{2q+s}}\right).
	\end{eqnarray*}	
	
	Summarizing  Case I and Case II, we have for any $-2q<s\leq p-2q$
	\begin{eqnarray}\label{z19}
		\|Ax(t)-b\|\leq\mathcal{O}\left(\frac{1}{t^{2q+s}}\right).
	\end{eqnarray}	
By $(ii)$ of Theorem \ref{ztt3.1}, we get 
	\begin{equation}\label{zhfa4}
\mathcal{L}_{t}(x(t),\lambda_t)-\mathcal{L}_{t}(x_t,\lambda_t)\leq\mathcal{O}\left(\frac{1}{t^{1-p+q+s}}\right).
\end{equation}
Since $0<p<1-q$ and $-2q<s\leq p-2q$, we have $2q+s<1-p+q+s$ and $2q+s\leq p$. It follows from \eqref{z15}, \eqref{z19} and \eqref{zhfa4} that
	$$\mathcal{L}(x(t),\lambda^*)-\mathcal{L}(x^*,\lambda^*)\leq\mathcal{O}\left(\frac{1}{t^{2q+s}}\right).$$
This together with \eqref{z16} and \eqref{z19} yields
	$$|f(x(t))-f(x^*)|\leq\mathcal{O}\left(\frac{1}{t^{2q+s}}\right).$$
   \end{proof}
	
\begin{remark}\label{remark2}
	When $0\leq q <1$, $\frac{1-q}{3}<p<1-q$ and $\frac{1-p-5q}{2}<s\leq p-2q$, it is easy to verify that
	$$-2q<s\leq p-2q, \quad \frac{p-3q-1}{2}<s<1-3q, \quad 2q+s>\min\{p, \frac{1-2q-s-r}{2}\}.$$
	Therefore, the convergence rate $\mathcal{O}\left(\frac{1}{t^{2q+s}}\right)$ of the primal-dual gap, the objective residual, the feasibility violation in Theorem \ref{ztt3.2} improves the convergence rate $\mathcal{O}\left(\frac{1}{t^{p}}+\frac{1}{t^{(1-2q-s-r)/2}}\right)$ established in $(ii)$ of Theorem \ref{ztt3.1} when $\frac{1-q}{3}<p<1-q$ and $\frac{1-p-5q}{2}<s\leq p-2q$.
\end{remark}
	
	By Theorem \ref{ztt3.1} and  Theorem \ref{ztt3.2}, we have the following corollary which improves the results of  Chbani et al. \cite{ChbaniRBOn(2024)}.
	
	\begin{corollary}\label{corollary1}
		Suppose that  Assumption \ref{AS-F} holds with $q=0$ and $p-1<s<1$. Let $(x(t),\lambda(t))$ be a solution of \eqref{z2} and $(x^*,\lambda^*)$ be the minimum norm element of  $\Omega$. Then, 
		$$\lim_{t\rightarrow+\infty}\|(x(t), \lambda(t))-(x^*,\lambda^*)\|=0$$ and  the following conclusions hold:
		\begin{itemize}
			\item[(i)] When $p-1<s<\frac{p-1}{2}$,  it holds
			$$\|x(t)-x_t\|^2\leq\mathcal{O}\left(\frac{1}{t^{2(1+s-p)}}\right), \quad \|\lambda(t)-\lambda_t\|^2\leq\mathcal{O}\left(\frac{1}{t^{2(1+s-p)}}\right),$$ $$\|\dot{x}(t)\|^2\leq\mathcal{O}\left(\frac{1}{t^{2(1+s-p)}}\right), \quad \|Ax(t)-b\|\leq \mathcal{O}\left(\frac{1}{t^p}+\frac{1}{t^{1+s-p}}\right).$$
			Further, if $\frac{2p-2}{3}<s<\frac{p-1}{2}$, then
			$$\mathcal{L}_{t}(x(t),\lambda_t)-\mathcal{L}_{t}(x_t,\lambda_t)\leq\mathcal{O}\left(\frac{1}{t^{3s-2p+2}}\right),$$
			$$\mathcal{L}(x(t),\lambda^*)-\mathcal{L}(x^*,\lambda^*)\leq\mathcal{O}\left(\frac{1}{t^{p}}+\frac{1}{t^{3s-2p+2}}\right),$$ $$|f(x(t))-f(x^*)|\leq\mathcal{O}\left(\frac{1}{t^{p}}+\frac{1}{t^{3s-2p+2}}\right).$$
			\item[(ii)] When $\frac{p-1}{2}\leq s<1$,  it holds
			$$\mathcal{L}_{t}(x(t),\lambda_t)-\mathcal{L}_{t}(x_t,\lambda_t)\leq\mathcal{O}\left(\frac{1}{t^{1-r}}\right), \quad \|x(t)-x_t\|^2\leq\mathcal{O}\left(\frac{1}{t^{1-s-r}}\right),$$ $$\|\lambda(t)-\lambda_t\|^2\leq\mathcal{O}\left(\frac{1}{t^{1-s-r}}\right), \quad \|\dot{x}(t)\|^2\leq\mathcal{O}\left(\frac{1}{t^{1-s-r}}\right).$$
			$$\mathcal{L}(x(t),\lambda^*)-\mathcal{L}(x^*,\lambda^*)\leq\mathcal{O}\left(\frac{1}{t^{p}}+\frac{1}{t^{(1-s-r)/2}}\right),$$
			$$|f(x(t))-f(x^*)|\leq\mathcal{O}\left(\frac{1}{t^{p}}+\frac{1}{t^{(1-s-r)/2}}\right), \quad \|Ax(t)-b\|\leq\mathcal{O}\left(\frac{1}{t^{p}}+\frac{1}{t^{(1-s-r)/2}}\right),$$
			where $r=\max\{0,p-s\}$.
			\item[(iii)] When $0<s\leq p$, it holds
$$\mathcal{L}(x(t),\lambda^*)-\mathcal{L}(x^*,\lambda^*)\leq\mathcal{O}\left(\frac{1}{t^{s}}\right),$$
	$$|f(x(t))-f(x^*)|\leq\mathcal{O}\left(\frac{1}{t^{s}}\right), \quad \|Ax(t)-b\|\leq\mathcal{O}\left(\frac{1}{t^{s}}\right).$$

		\end{itemize}
	\end{corollary}
	\begin{proof}
Items (i) and (ii) follow directly from Theorem \ref{ztt3.1}, and item (iii) follows directly from  Theorem \ref{ztt3.2}.
\end{proof}

	\begin{remark}\label{remark4}
When $s=p$,  Corollary \ref{corollary1} recovers  the convergence rate results of  \cite[Theorem 3.2]{ChbaniRBOn(2024)} where the condition
 $\frac{1}{\theta}<\alpha<\frac{1}{\theta}+\min(\frac{1}{\theta}, c)$ and either $\alpha<2\sqrt{c}$ or $2\sqrt{c}<\alpha<\frac{1}{\theta}+c\theta$ was assumed, instead of the condition $\frac{1}{\theta}<\alpha$ used in  Corollary \ref{corollary1}. It is worth mentioning that the proof of  \cite[Theorem 3.2]{ChbaniRBOn(2024)} was based on  \cite[Lemma 2.1]{ChbaniRBOn(2024)} (\cite[Lemma 6]{HeHFiietal(2022)}), which cannot be  applied there since the  function $a(s)$ is dependent on $t$. To fix this, we develop Lemma  \ref{lemma2.2} and Lemma  \ref{lemma2.3} to establish the convergence rate results in  Theorem \ref{ztt3.2}.
	\end{remark}


	
	\section{Numerical experiments}

	In this section, we perform some numerical experiments to illustrate the theoretical results on our dynamical system \eqref{z2}. All codes are run on a PC (with 2.20GHz Dual-Core Intel Core i7 and 16GB memory) under MATLAB Version R2017b and  all the
dynamical systems are solved numerically by the ode23 in MATLAB.
	
	Consider the linearly constrained convex optimization problem
	\begin{eqnarray}\label{zyc1}
		\min_{x\in\mathbb{R}^3} (x_1-x_2)^2+x_{3}^2, \quad\text{ s.t. }x_1-x_2+x_3-2=0,
	\end{eqnarray}
	where $x=(x_1,x_2,x_3)^T$. Then,  $f(x)= (x_1-x_2)^2+x_{3}^2, A=(1,-1,1)$ and $b=2$. By  means of \eqref{zc1}, it is easy to verify that 
$$\Omega=\{(x,\lambda): x_1-x_2=1, x_3=1,\lambda=-2\},$$
 $x^*=\left(\frac{1}{2},-\frac{1}{2},1\right)^T$ is the minimal norm solution of problem \eqref{zyc1},  $(x^*,-2)$ is the minimal norm element of $\Omega$, and  $f^*=2$ is the optimal objective function value of  problem \eqref{zyc1}. Because $\lambda^*=-2$ is the unique dual solution of problem \eqref{zyc1}, we only display  numerical results on the associated trajectories involving the primal trajectory $x(t)$ in the following  numerical experiments. In what follows, we always take the starting points $x(1)=(1,-1,1)^T$, $\lambda(1)=(1)$, $\dot{x}(1)=(1,1,1)^T$ in our dynamical system \eqref{z2}.

	In the first numerical experiment,  take $\theta=1$, $\alpha=3$, $c=0.1$, $q=0$, $s=p=\{0.2, 0.5, 0.7, 0.9\}$ in the proposed dynamical system \eqref{z2}.  In this setting on the parameters, all assumptions in Theorem \ref{ztt3.2} hold, but the conditions $\frac{1}{\theta}<\alpha<\frac{1}{\theta}+\min(\frac{1}{\theta}, c)$ and either $\alpha<2\sqrt{c}$ or $2\sqrt{c}<\alpha<\frac{1}{\theta}+c\theta$ imposed in  \cite[Theorem 3.2]{ChbaniRBOn(2024)} are not satisfied. Figure \ref{fig:titstfig} shows that the behaviors of  $\|x(t)-x^*\|$, $|f(x(t))-f^*|$, and  $\|Ax(t)-b\|$ under different choices of $p\in(0,1)$.
	\begin{figure*}[h]
		\centering
		{
			\begin{minipage}[t]{0.321\linewidth}
				\centering
				\includegraphics[width=2.0in]{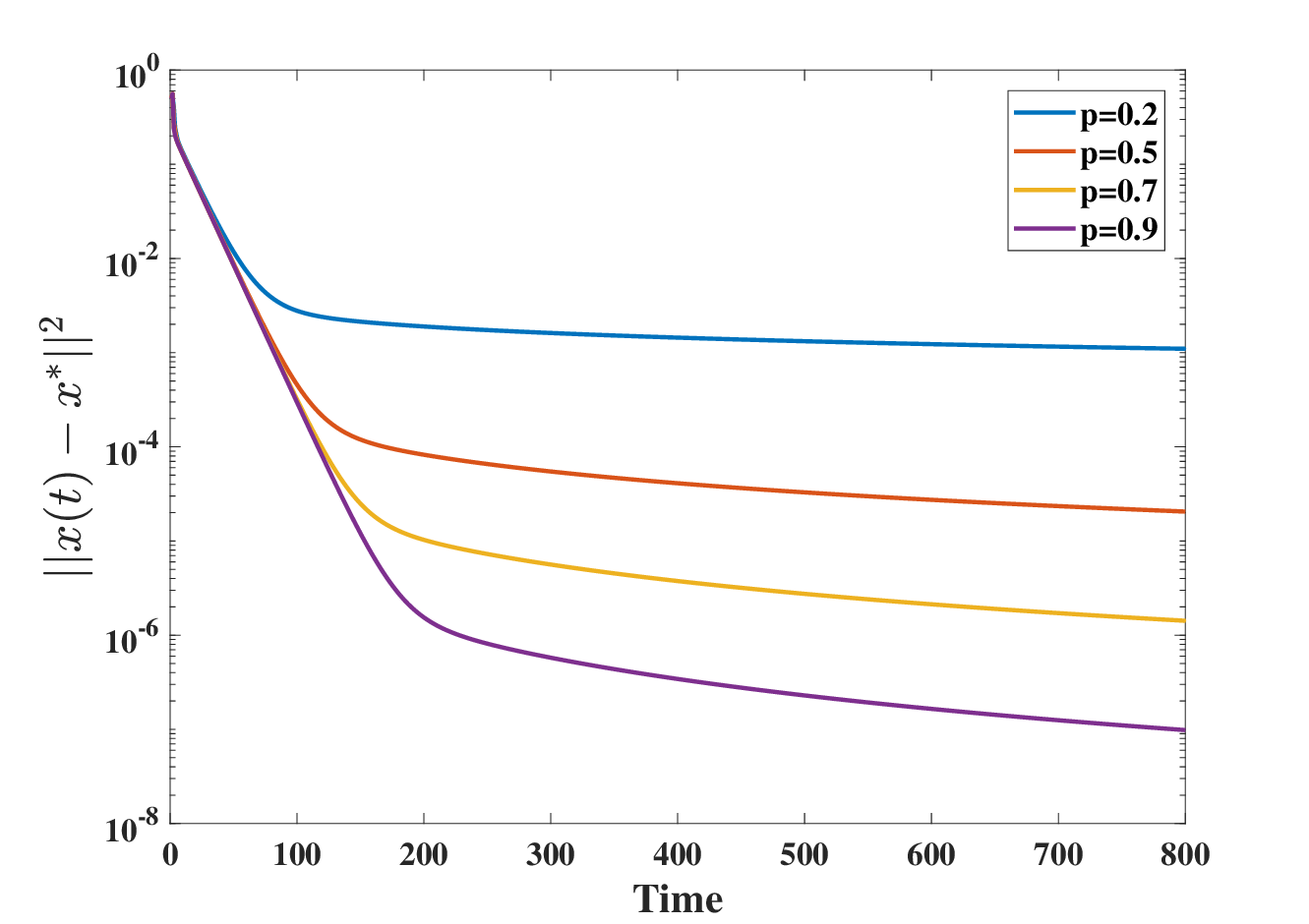}
			\end{minipage}%
		}
		{
			\begin{minipage}[t]{0.321\linewidth}
				\centering
				\includegraphics[width=2.0in]{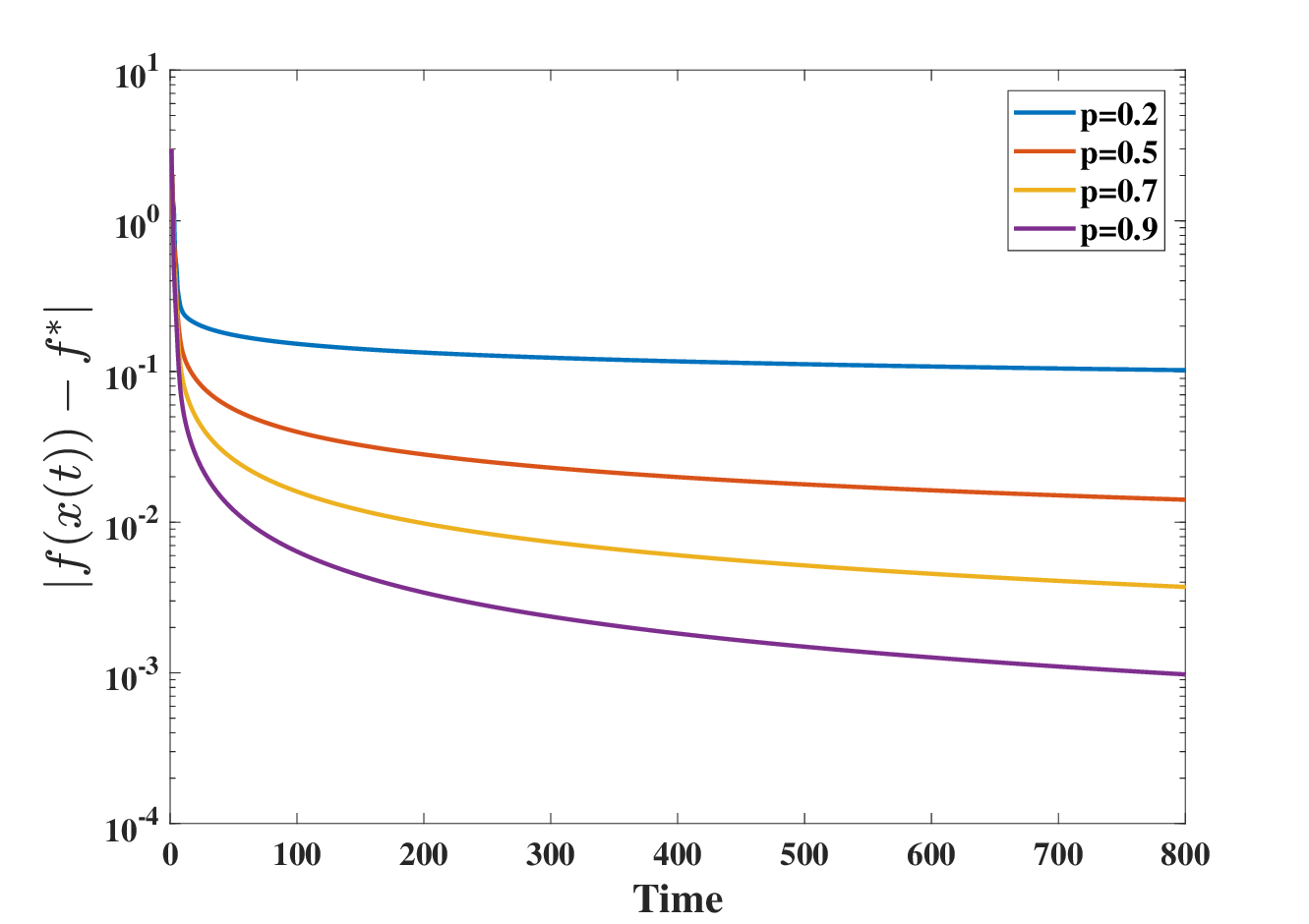}
			\end{minipage}%
		}
		{
			\begin{minipage}[t]{0.321\linewidth}
				\centering
				\includegraphics[width=2.0in]{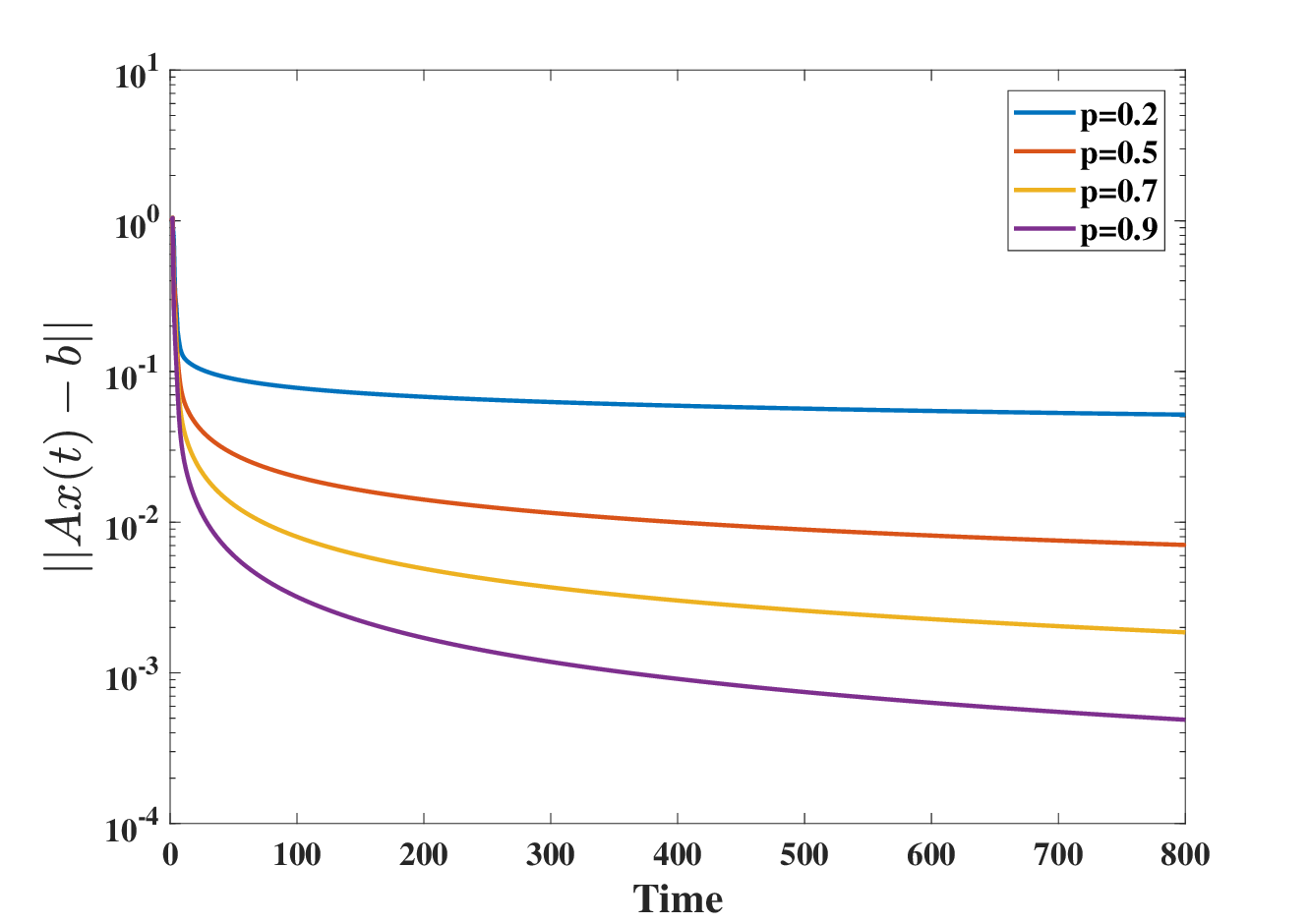}
			\end{minipage}%
		}
		\caption{Error analysis of system \eqref{z2} with $q=0$ and $s=p$ for problem \eqref{zyc1}  with parameter setting which cannot guarantee the conditions of \cite[Theorem 3.2]{ChbaniRBOn(2024)}}
		\label{fig:titstfig}
		\centering
	\end{figure*}

	In the second numerical experiment, we compare our dynamical system \eqref{z2} with $\text{(He-ODE)}$ in \cite{HHFIPD2023} under the different choices of $s$. Take $\theta=1$, $\alpha=3$, $c=0.1$, $q=0.1$, $p=0.6$ and $s\in\{0.15, 0.4, 0.65\}$ in system \eqref{z2} and take $\theta=1$, $\alpha=3$, $\rho=1$, $k=q=0.1$, $\varepsilon(t)=0$ and $\beta(t)=t^s$ with $s\in\{0.15, 0.4, 0.65\}$ in $\text{(He-ODE)}$ \cite{HHFIPD2023}.  For  system \eqref{z2} and  $\text{(He-ODE)}$, we take the same starting points $x(1)=(1,-1,1)^T$, $\lambda(1)=(1)$, $\dot{x}(1)=(1,1,1)^T$, $\dot{\lambda}(1)=(1)$.
 
	\begin{figure*}[h]
		\centering
		\subfloat[Dynamical system \eqref{z2}]
		{
			\begin{minipage}[b]{0.321\linewidth}
				\centering
				\includegraphics[width=2.0in]{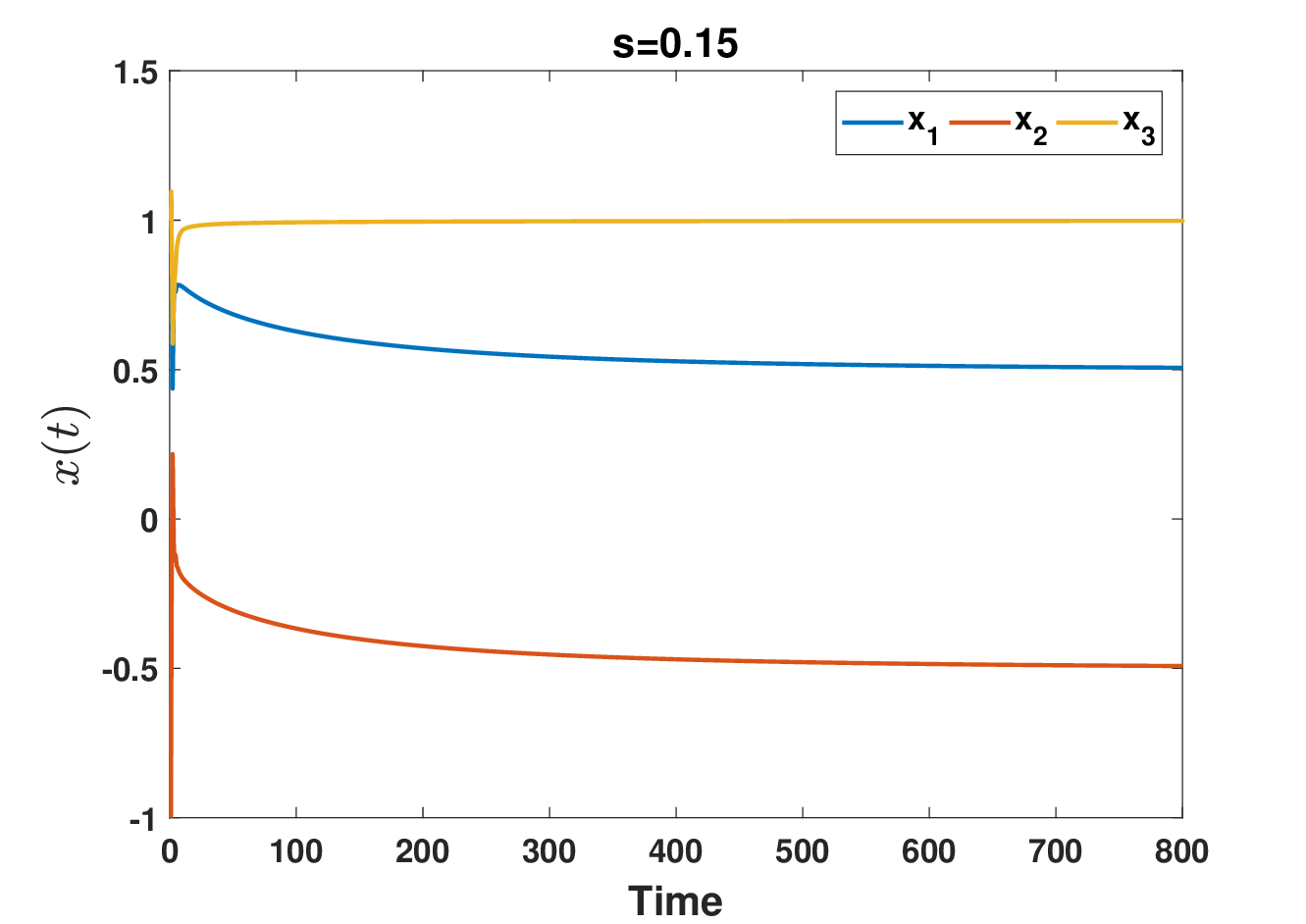}
			\end{minipage}
			\begin{minipage}[b]{0.321\linewidth}
				\centering
				\includegraphics[width=2.0in]{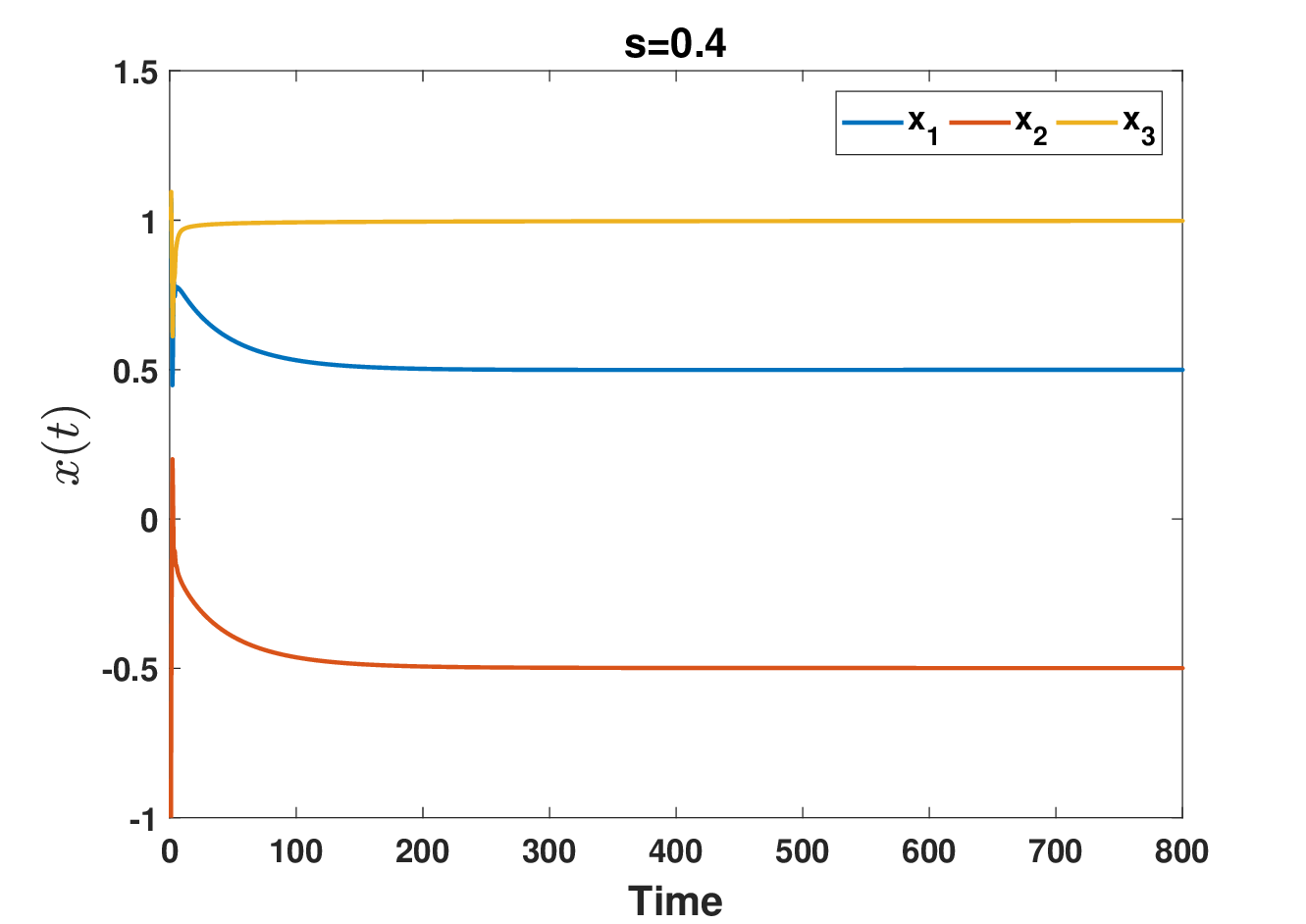}
			\end{minipage}
			\begin{minipage}[b]{0.321\linewidth}
				\centering
				\includegraphics[width=2.0in]{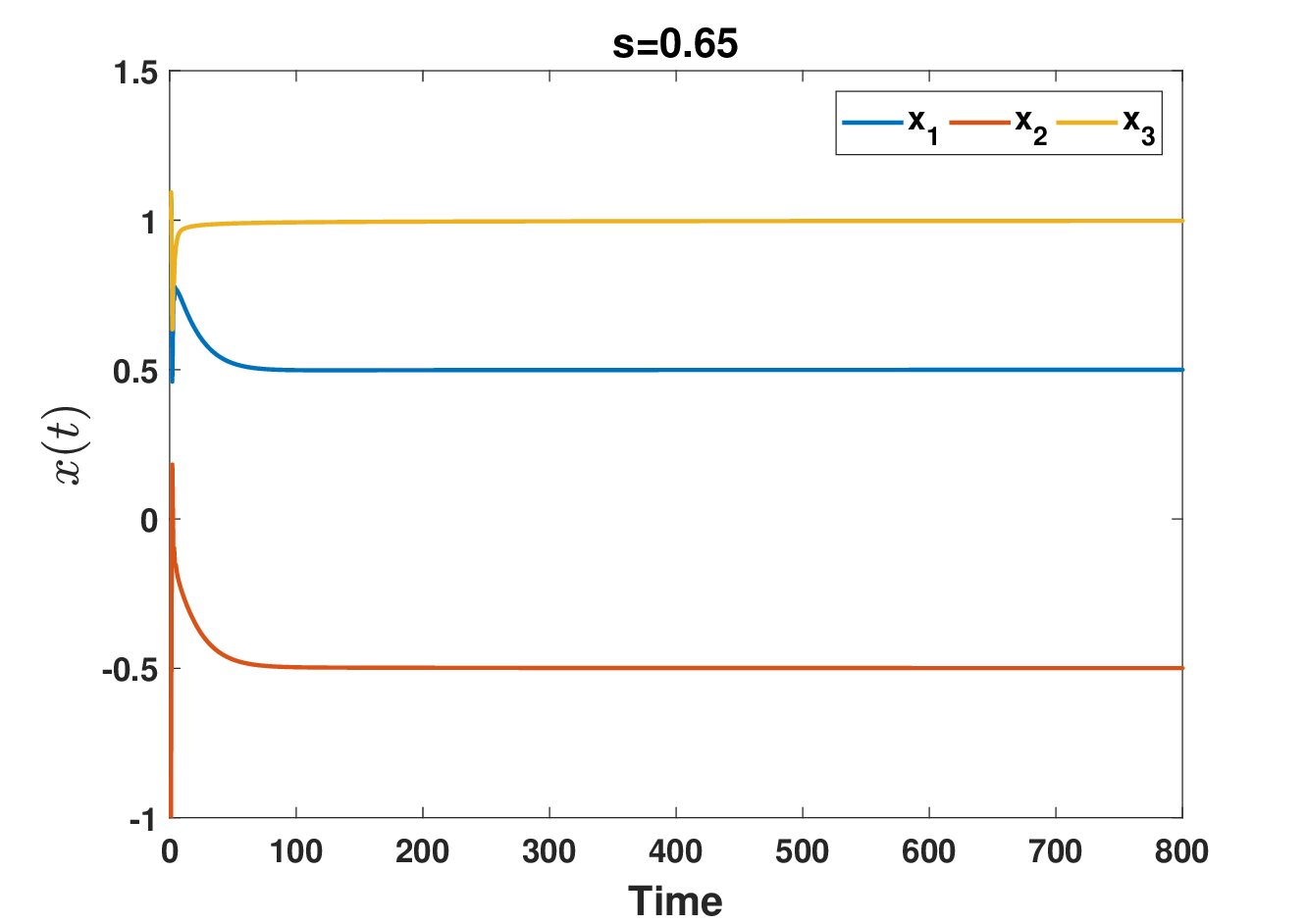}
			\end{minipage}
		}\\
		\subfloat[(He-ODE)]
		{
			\begin{minipage}[b]{0.321\linewidth}
				\centering
				\includegraphics[width=2.0in]{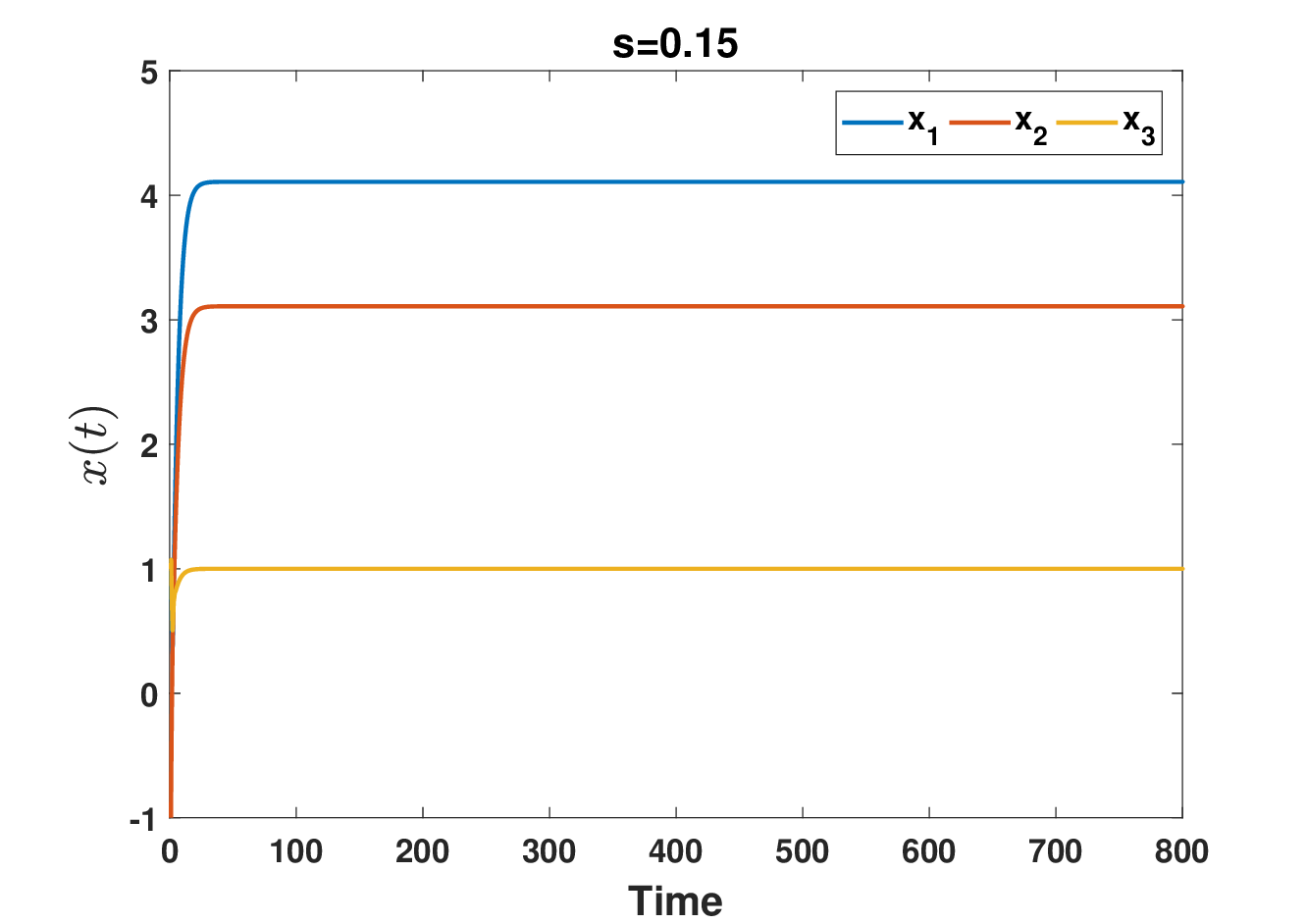}
			\end{minipage}
			\begin{minipage}[b]{0.321\linewidth}
				\centering
				\includegraphics[width=2.0in]{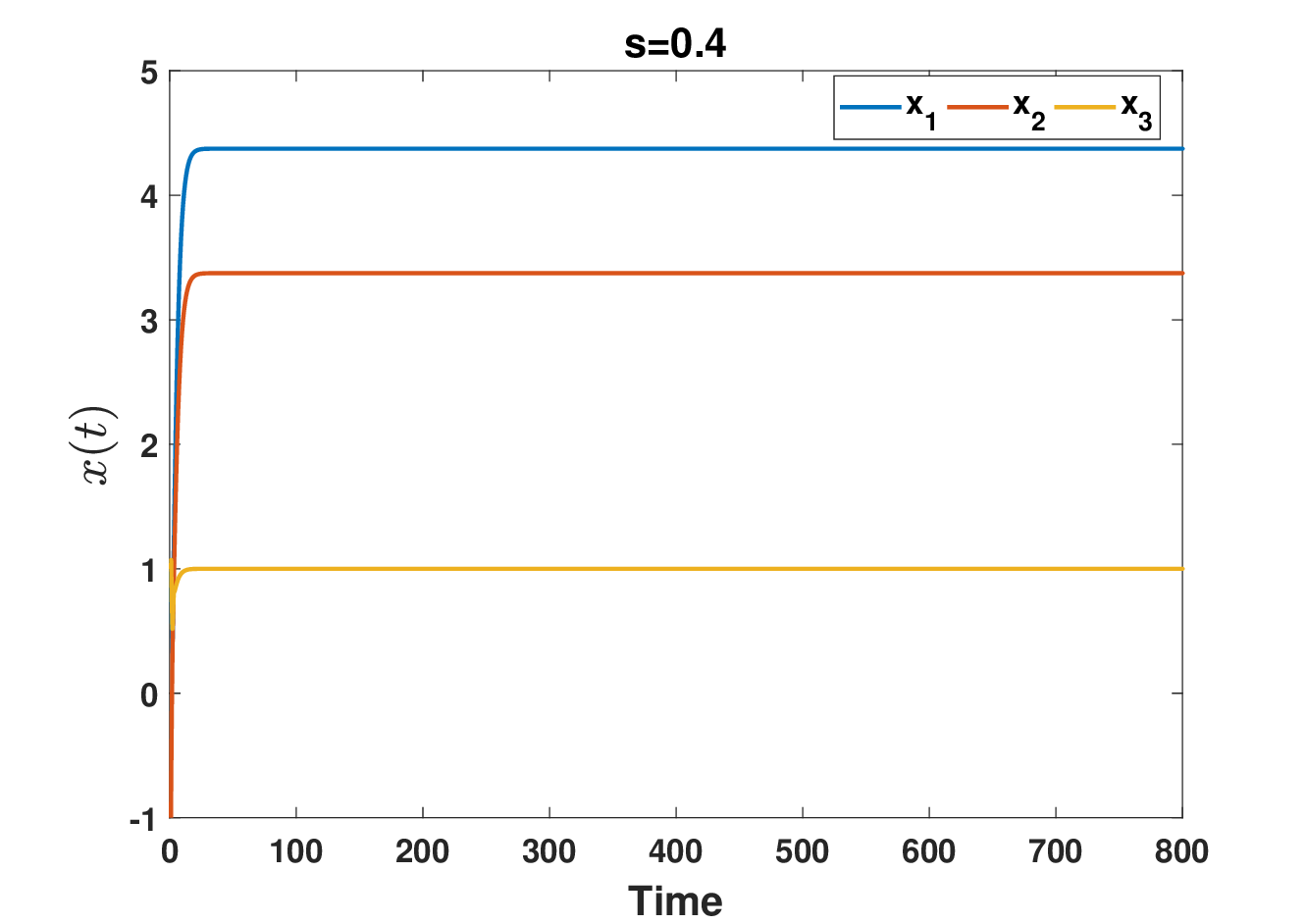}
			\end{minipage}
			\begin{minipage}[b]{0.321\linewidth}
				\centering
				\includegraphics[width=2.0in]{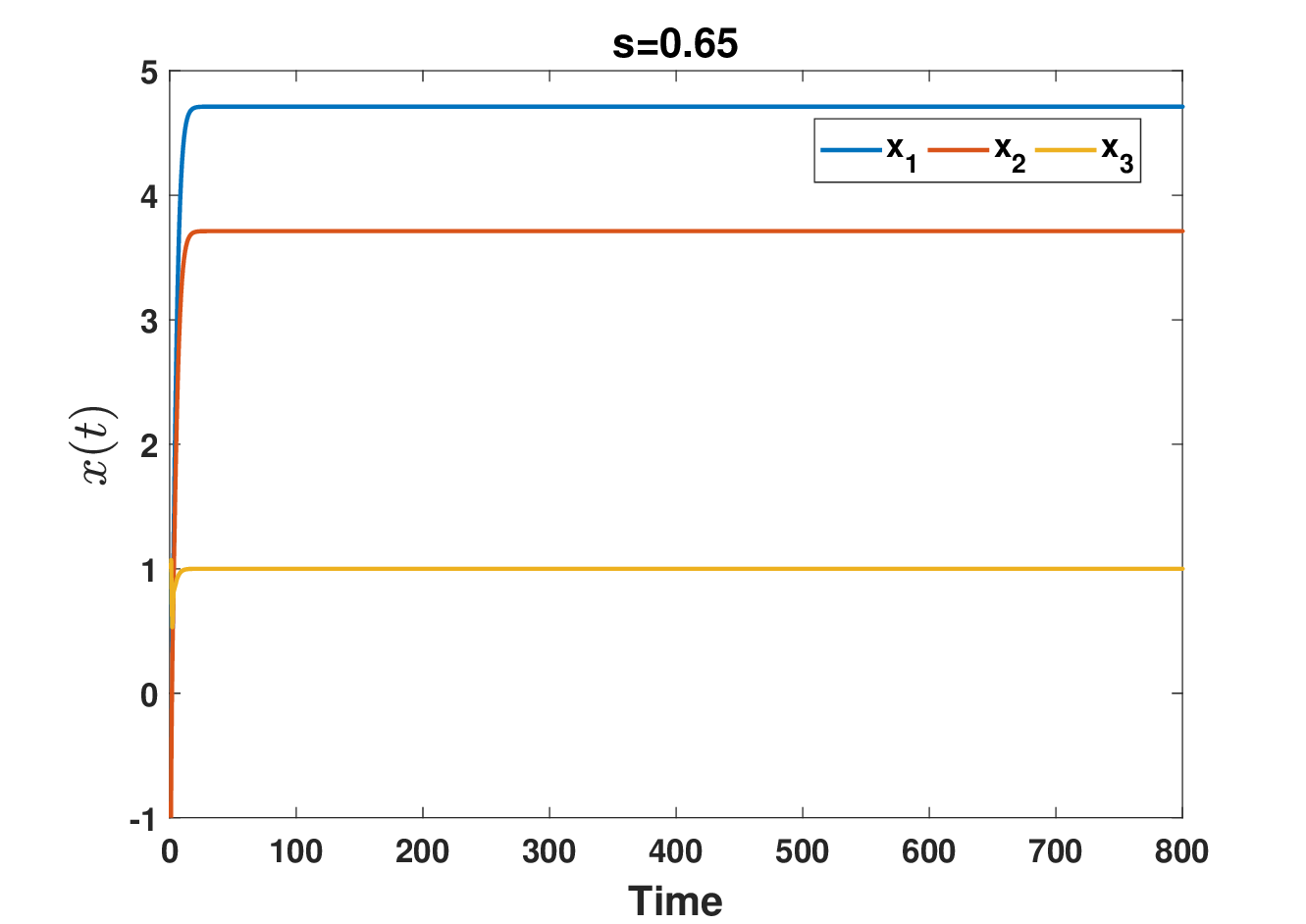}
			\end{minipage}
		}
		\caption{Comparisons of system \eqref{z2} and  $\text{(He-ODE)}$  in the the primal trajectory $x(t)$}
		\label{fig:titstfig2}
		\centering
	\end{figure*}
	
	As shown in  Figure \ref{fig:titstfig2},  the trajectory $x(t)$ of system \eqref{z2} converges to the minimal norm solution $x^*$ of problem \eqref{zyc1}, while the trajectory $x(t)$ of $\text{(He-ODE)}$ converges  to a solution of of problem \eqref{zyc1} which need not to be the minimal norm  solution $x^*$. 
	
	In the third numerical experiment, we display the behaviors of  $\|x(t)-x^*\|$, $|f(x(t))-f^*|$, and  $\|Ax(t)-b\|$ along the trajectory of system \eqref{z2}  under the different choices of the parameters $q$, $p$ and $s$. Take $\theta=1$, $\alpha=3$, $c=0.1$, $q=0.1$, $q\in\{0, 0.1\}$, $p\in\{0.2,0.4,0.6,0.8\}$ and $s\in\{-0.35,0.35,0.55,0.85\}$ in system \eqref{z2}. Figure \ref{fig:titstfig3} shows  the numerical results support  the theoretical results  in Theorem \ref{ztt3.1} and  Theorem \ref{ztt3.2}.

	\begin{figure*}[h]
		\centering
		{
			\begin{minipage}[t]{0.321\linewidth}
				\centering
				\includegraphics[width=2.0in]{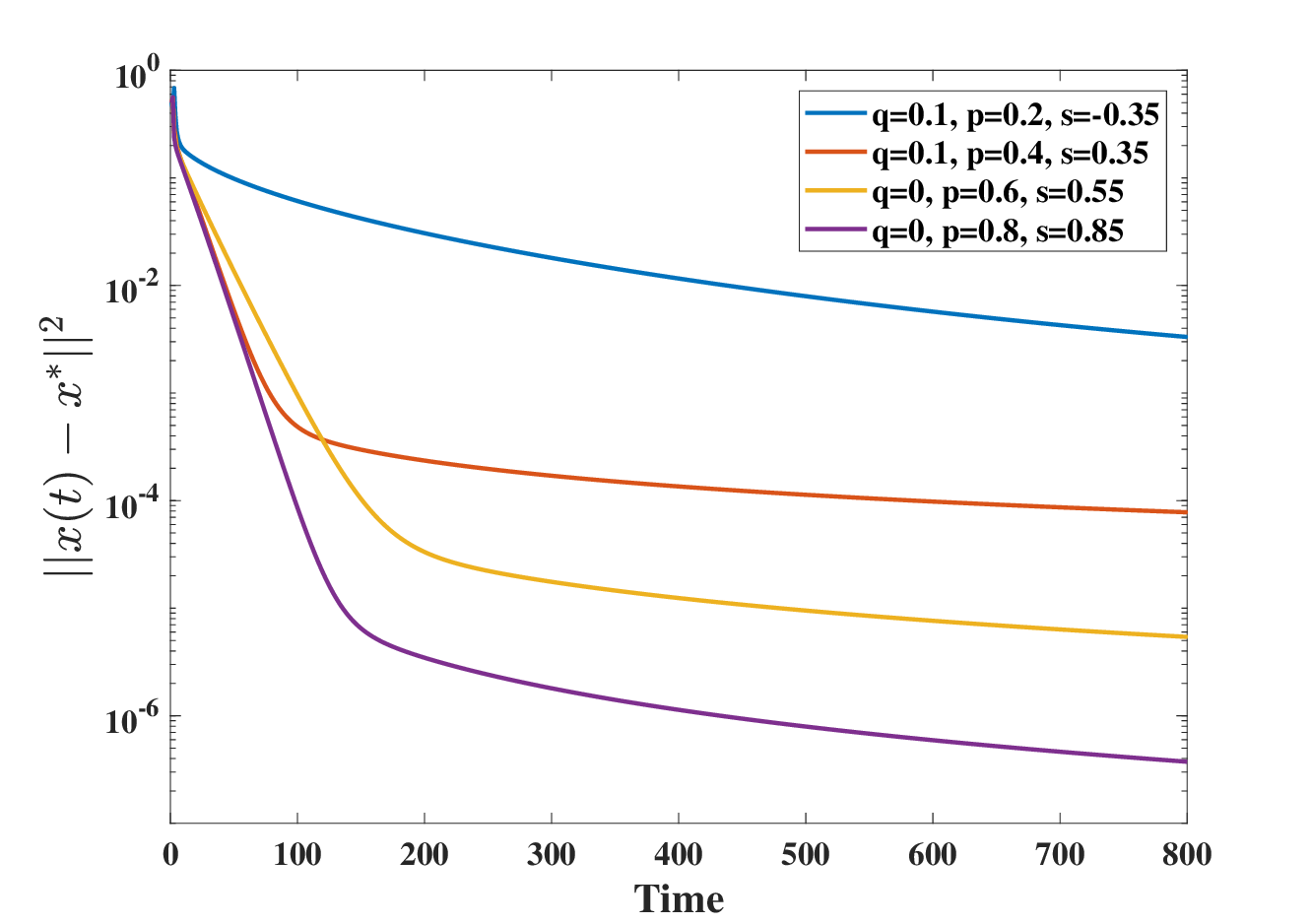}
			\end{minipage}%
		}
		{
			\begin{minipage}[t]{0.321\linewidth}
				\centering
				\includegraphics[width=2.0in]{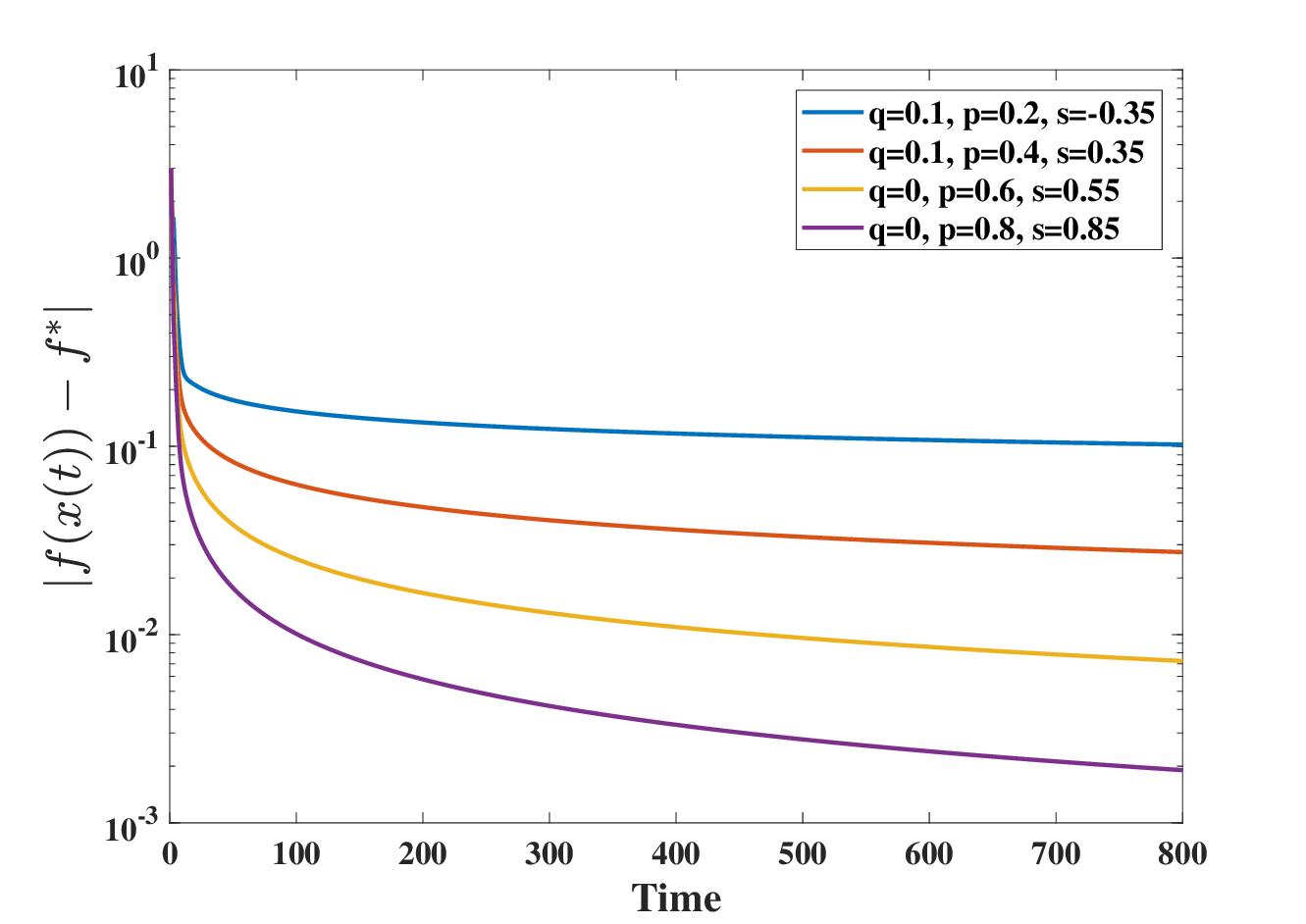}
			\end{minipage}%
		}
		{
			\begin{minipage}[t]{0.321\linewidth}
				\centering
				\includegraphics[width=2.0in]{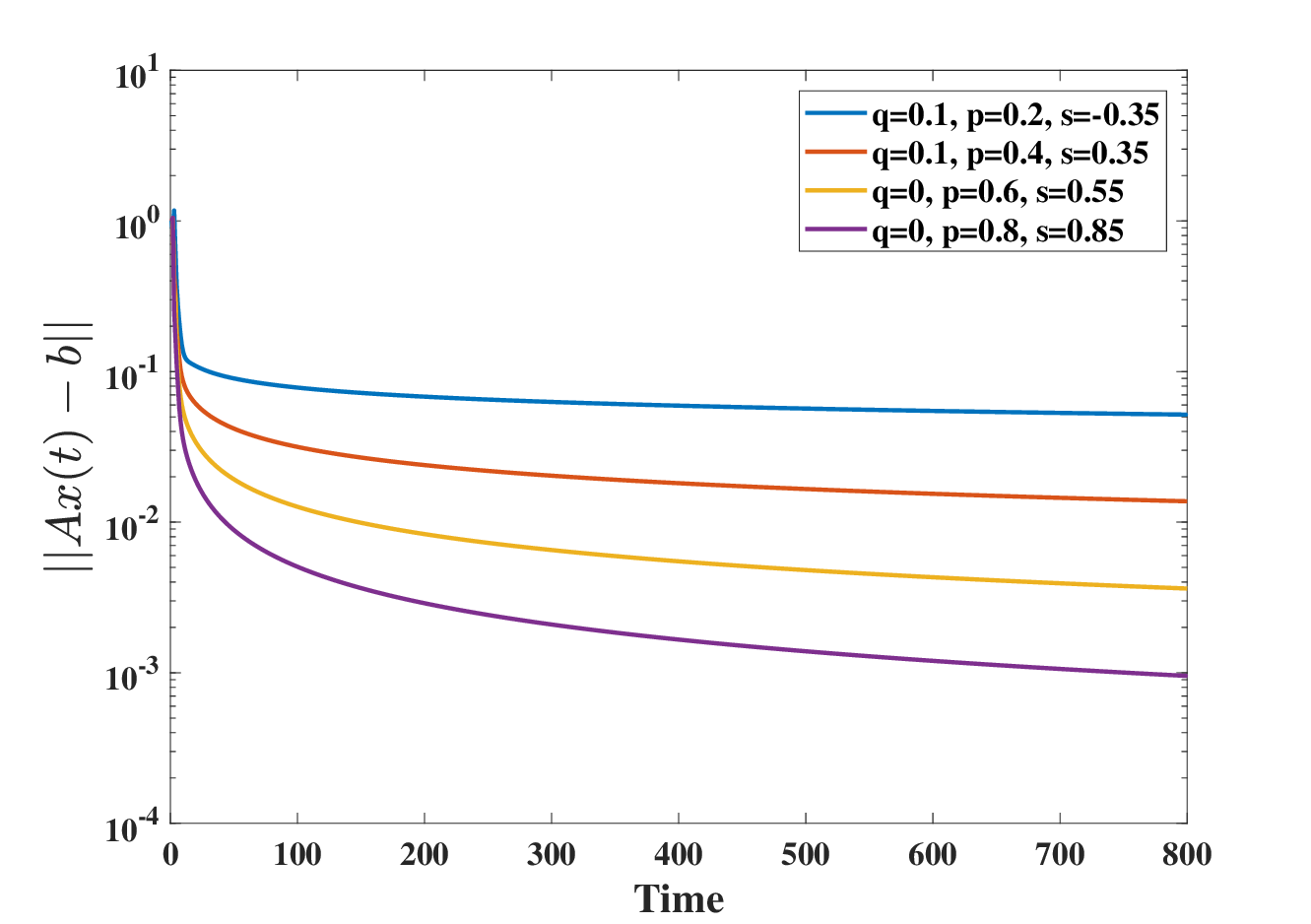}
			\end{minipage}%
		}
		\caption{Error analysis of dynamical system \eqref{z2}  with different $q$, $p$ and $s$ for problem \eqref{zyc1}}
		\label{fig:titstfig3}
		\centering
	\end{figure*}

\section{Conclusion}

In a Hilbert setting, we develop  a second-order plus first-order  primal-dual dynamical system, controlled by a Tikhonov regularization term, with a slow damping $\frac{\alpha}{t^q}$ with $0\le q<1$ and prove the strong convergence  of the trajectory of the proposed dynamical system to the minimal norm primal-dual solution, along with convergence rate results of the primal-dual gap, the objective residual and the feasibility violation. Let us emphasize that  the proofs of  the convergence rates of the objective residual and the feasibility violation are based on  new developed Lemma  \ref{lemma2.2} and Lemma  \ref{lemma2.3}, by which we can fill the gap in the proof of \cite[Theorem 3.2]{ChbaniRBOn(2024)}.

	


\begin{thebibliography}{}
		
		\bibitem{ZLinandLiHandFang(2020)}
		Lin ZC, Li H, Fang C. Accelerated optimization for machine learning. Nature Singapore: Springer; 2020.
		
		\bibitem{GoldsteinTandDonoghue(2014)}
		Goldstein T, O'Donoghue B, Setzer S, Baraniuk, R. Fast alternating direction optimization methods. SIAM Journal on Imaging Sciences. 2014;7(3):1588-1623.	
		
		\bibitem{ZengXLandLeiJLandChenJ(2022)}
		Zeng XL, Lei J, Chen J. Dynamical primal-dual Nesterov accelerated method and its application to network optimization. IEEE Transactions on Automatic Control. 2023;68(3):1760-1767.	
		
		\bibitem{PYiandHongYandLiu(2015)}
		Yi P, Hong YG, Liu F. Distributed gradient algorithm for constrained optimization with application to load sharing in power systems. Systems \& Control Letters. 2015;83:45-52.
		
          \bibitem{Cabot} Cabot A,  Engler H,  Gadat S. On the long time behavior of second order differential equations with asymptotically small dissipation. Transactions of the American Mathematical Society.  2009; 361: 5983–6017.
		
		\bibitem{Polyak(1964)}
		Polyak BT. Some methods of speeding up the convergence of iteration methods. USSR Computational Mathematics and Mathematical Physics. 1964;4(5):1-17.
		
		\bibitem{Alvarezon(2000)}
		Alvarez F. On the minimizing property of a second order dissipative system in Hilbert spaces. SIAM Journal on  Control and Optimization. 2000;38(4):1102–1119.
		
		\bibitem{Begoutbj(2015)}
		B{\'e}gout P, Bolte J, Jendoubi MA. On damped second-order gradient systems. Journal of Differential Equations.  2015;259(7):3115–3143.
		
		\bibitem{SuBoydandCandes(2016)}
		Su WJ, Boyd S, Cand{\`e}s E. A differential equation for modeling Nesterov’s accelerated gradient method: Theory and insights. The Journal of Machine Learning Research. 2016;17:5312-5354.


		\bibitem{Nesterov(1983)}
		Nesterov Y. A method of solving a convex programming problem with convergence rate $\mathcal{O}(\frac{1}{k^2})$. Sov. Math. Dokl. 1983;27:372-376.
		
		\bibitem{Nesterov(2013)}
		Nesterov Y. Gradient methods for minimizing composite functions. Mathematical Programming. 2013;140(1):125-161.

 	\bibitem{AttouchCPR2018}  
 	Attouch H,  Chbani Z,  Peypouquet J, Redont P.  Fast convergence of inertial dynamics and algorithms with asymptotic vanishing viscosity. Mathematical Programming. 2018; {168}(1-2):123-175.
 	
 	\bibitem{May2017} 
 	May R. Asymptotic for a second-order evolution equation with convex potential and vanishing damping term. Turkish J. Math. 2017; {41}(3):681-685.
	
 	\bibitem{AttouchCRR2019} 
 	Attouch H, Chbani Z, Riahi H. Rate of convergence of the Nesterov accelerated gradient method in the subcritical case $\alpha\leq 3$. ESAIM: Control Optim. Calc. Var.  2019; {25}. Article Number: 2.

 	\bibitem{Vassilis2018} Vassilis A, Jean-Fran{\c{c}}ois A, Charles D. The differential inclusion modeling FISTA algorithm and optimality of convergence rate in the case $b \leq3$. SIAM J. Optim. 2018; {28}(1):551-574.

 		
	\bibitem{Aujol2019}
	 Aujol JF, Dossal C, Rondepierre A.  Optimal convergence rates for Nesterov acceleration. SIAM J. Optim. 2019; {29}(4):3131-3153.
 	


\bibitem{Cabjde}
Cabot A, Frankel P. Asymptotics for some semilinear hyperbolic equations with non-autonomous damping. J. Differ. Equ. 2012;252:294–322.

\bibitem{Hara}
Haraux A, Jendoubi MA. Asymptotics for a second order differential equation with a linear, slowly time-decaying damping term. Evol. Eqs. Control Theory. 2013;2(3):461–470.

\bibitem{Balti2017}
Balti M, May R. Asymptotic for the perturbed heavy ball system with vanishing damping term. Evol. Equ. Control Theory. 2017;6(2):177–186.

\bibitem{Attouchcabot(2017)}
Attouch H, Cabot A. Asymptotic stabilization of inertial gradient dynamics with time-dependent viscosity. J Differ Equ. 2017;263(9):5412–5458.

\bibitem{Sebb}
Sebbouh O, Dossal C, Rondepierre A. Convergence rates of damped inertial dynamics under geometric conditions and perturbations. SIAM J. Optim. 2020;30(3):1850–1877.

\bibitem{GeB} 
Ge B, Zhuge X, Ren H. Convergence rates of damped inertial dynamics from
multi-degree-of-freedom system. Optimization Letters. (2022) 16:2753–2774. 



		
		
		
		
		
		
		
		
		%
		
		
		\bibitem{AttouchandCzarnecki(2002)}
		Attouch H, Czarnecki MO. Asymptotic control and stabilization of nonlinear oscillators with non-isolated equilibria. Journal of Differential Equations. 2002;179(1):278-310.
		
		\bibitem{AttouchZH2018}
		Attouch H, Chbani Z, Riahi H. Combining fast inertial dynamics for convex optimization with Tikhonov regularization. Journal of Mathematical Analysis and Applications. 2018;457(2):1065-1094.
		
		\bibitem{Attouchlaszlo2021}
		Attouch H, László SC. Convex optimization via inertial algorithms with vanishing Tikhonov regularization: Fast convergence to the minimum norm solution. 2021, https://arxiv.org/abs/2104.11987.

		
		
		\bibitem{AttouchBCR2022311}
		Attouch H, Balhag A, Chbani Z, Riahi H, Damped inertial dynamics with vanishing Tikhonov regularization: strong asymptotic convergence towards the minimum norm solution. J. Differ. Equ. 2022;311:29–58.
		
		\bibitem{Laszlo2023}
		L{\'a}szl{\'o} SC. On the strong convergence of the trajectories of a Tikhonov regularized second order dynamical system with asymptotically vanishing damping. Journal of Differential Equations. 2023;362:355-381.

		\bibitem{botcsernek2021}
		 Bo{\c t} RI, Csetnek ER, László SC, Tikhonov regularization of a second order dynamical system with Hessian damping. Math. Program. 2021;189(1-2):151–186.
		
		\bibitem{Alecsalaszlo2021}
		Alecsa CD, László SC. Tikhonov regularization of a perturbed heavy ball system with vanishing damping. SIAM J. Optim. 2021;31(4):2921–2954.



		
		\bibitem{HeHuFangetal(2021)}
		He X, Hu R, Fang YP. Convergence rates of inertial primal-dual dynamical methods for separable convex optimization problems. SIAM Journal on Control and Optimization. 2021;59(5):3278-3301.
		
		\bibitem{AttouchADMM(2022)}
		Attouch H, Chbani Z, Fadili J,  Riahi H. Fast convergence of dynamical ADMM via time scaling of damped inertial dynamics. Journal of Optimization Theory and Applications. 2022;193:704-736.
		
		\bibitem{BNguyen2022}
		 Bo{\c t} RI, Nguyen DK. Improved convergence rates and trajectory convergence for primal-dual dynamical systems with vanishing damping. Journal of Differential Equations. 2021;303:369-406.
		
		\bibitem{HHFIPD2023}
		He X, Hu R, Fang YP. Inertial primal-dual dynamics with damping and scaling for linearly constrained convex optimization problems. Applicable Analysis. 2023;102(15):4114-4139.
		
		
		\bibitem{HeHFetal(2022)}
		He X, Hu R, Fang YP. ``Second-order primal"+``first-order dual" dynamical systems with time scaling for linear equality constrained convex optimization problems. IEEE Transactions on Automatic Control. 2022;67(8):4377-4383.
		
		\bibitem{HeHFiietal(2022)}
		He X, Hu R, Fang YP. Fast primal–dual algorithm via dynamical system for a linearly constrained convex optimization problem. Automatica. 2022;146:110547.
		
		
		

\bibitem{zhuhufang1}	
		Zhu TT, Hu R, Fang YP. Tikhonov regularized second-order plus first-order primal-dual dynamical systems with asymptotically vanishing damping for linear equality constrained convex optimization problems. Preprint  arXiv (2023). http://arxiv.org/abs/2307.03612v2.

\bibitem{zhuhufang2}	
		Zhu TT, Hu R, Fang YP. Fast convergence rates and trajectory convergence of a Tikhonov regularized inertial primal-dual dynamical system with time scaling and vanishing damping. Preprint arXiv (2024). https://doi.org/10.48550/arXiv.2404.14853.
		
		\bibitem{ChbaniRBOn(2024)}
		Chbani Z, Riahi H, Battahi F. On the simultaneous convergence of values and trajectories of continuous inertial dynamics with Tikhonov regularization to solve convex minimization with affine constraints. Preprint HAL (2024). https://hal.science/hal-04511296.
		
				
		
		\bibitem{HeTLFconver(2023)}
		He X, Tian F, Li AQ, Fang YP. Convergence rates of mixed primal-dual dynamical systems with Hessian driven damping. Optimization.(2023) .https://doi.org/10.1080/02331934.2023.2253813.
		
		
		
		
		
		
		
	\end{thebibliography}
\end{document}